\documentclass[english,11pt,reqno]{amsart}
\usepackage{babel}
\usepackage{amsmath,amsfonts,amssymb}

\makeatletter
\newsavebox{\@brx}
\newcommand{\llangle}[1][]{\savebox{\@brx}{\(\m@th{#1\langle}\)}%
  \mathopen{\copy\@brx\kern-0.5\wd\@brx\usebox{\@brx}}}
\newcommand{\rrangle}[1][]{\savebox{\@brx}{\(\m@th{#1\rangle}\)}%
  \mathclose{\copy\@brx\kern-0.5\wd\@brx\usebox{\@brx}}}
\makeatother

\newcommand{\twoheaddownarrow}{\raise2pt\hbox{%
  \ooalign{\hss$\downarrow$\hss\cr\lower2pt\hbox{%
  $\downarrow$}}}}
\newcommand{\twoheaduparrow}{\raise2pt\hbox{%
  \ooalign{\hss$\uparrow$\hss\cr\lower2pt\hbox{%
  $\uparrow$}}}}

\theoremstyle{plain}
\newtheorem{theorem}{Theorem}[section]
\newtheorem{proposition}[theorem]{Proposition}
\newtheorem{lemma}[theorem]{Lemma}
\newtheorem{corollary}[theorem]{Corollary}

\theoremstyle{definition}
\newtheorem{definition}[theorem]{Definition}
\newtheorem{example}[theorem]{Example}
\newtheorem{assumption}[theorem]{Assumption}

\theoremstyle{remark}
\newtheorem{remark}[theorem]{Remark}

\numberwithin{equation}{section}

\begin{document}
\title[Scott functions and their representations]
{Scott functions, their representations on domains, and applications
  to random sets}

\author{Motoya Machida}
\address{Department of Mathematics, Tennessee Technological University,
Cookeville, Tennessee 38505}
\email{mmachida@tntech.edu}

\author{Alexander Y.~Shibakov}
\address{Department of Mathematics, Tennessee Technological University,
Cookeville, Tennessee 38505}
\email{ashibakov@tntech.edu}

\date{\today}

\begin{abstract}
Choquet theorems (1954) on integral representation for capacities
are fundamental to probability theory. They inspired
a growing body of research into different approaches
and generalizations of Choquet's results by many other researchers.
Notably Math\'{e}ron's work (1975) on distributions over the space
of closed subsets has led to further advancements in the theory of random sets.

This paper was inspired by the work of Norberg
(1989) who generalized Choquet's results to distributions over domains.
While Choquet's original theorems were obtained for locally compact
Hausdorff (LCH) spaces,
both Math\'{e}ron's and Norberg's depend on the assumption
of separability in their application of the Carath\'{e}odory's method.

Our Radon measure approach differs from the work of Math\'{e}ron
and Norberg, in that it does not require separability.
This investigation naturally leads to
the introduction of finite and locally finite valuations,
which allows us to characterize
finite and locally finite random sets
in terms of capacities on the class of compact subsets.

Finally, the treatment of L\'{e}vy exponent by Math\'{e}ron
and Norberg is revisited,
and the notion of exponential valuation is proposed for the
representation of general Poisson processes.
\end{abstract}

%%% 60D05 Geometric probability and stochastic geometry
%%% 28C15 Set functions and measures on topological spaces
%%%       (regularity of measures, etc.)
%%% 06B35 Continuous lattices and posets, applications
%%%       [See also 06B30, 06D10, 06F30, 18B35, 22A26, 68Q55]


\maketitle

\section{Introduction}

\subsection{Capacities and Choquet theorems}
\label{choquet.intro}
The theory of {\em random sets\/} forms an important area of probability
that has generated a lot of research activity
(see~\cite{molchanov} for an extensive survey of the field and the references therein)
dating back to the
foundational work of Kolmogorov~\cite{kolmogorov} in the 1930s.
An important tool in the analysis of the
distribution of a given random set is supplied by the {\em capacity\/}
functional that translates the study of the (very rich)
$\sigma$-algebra of random set into the investigation of measure-like
functionals on the sets themselves.

The capacity functional is not a direct replacement of a measure due
to its nonadditivity. Nonetheless, the classical results of Choquet
show that a few natural algebraic properties are all that is required
to establish a one-to-one correspondence between random set
distributions and capacities.

Let us introduce the notion of capacities over the class
$\mathcal{K}$
of compact subsets of a locally compact Hausdorff (LCH) space~$R$.
A set function $\Phi$ on $\mathcal{K}$ is called
\emph{continuous on the right} in \cite{choquet}
if for any $E\in\mathcal{K}$ and $\varepsilon > 0$
there is an open neighborhood $G$ of $E$ such that
$|\Phi(F) - \Phi(E)| < \varepsilon$
whenever $E\subseteq F\subseteq G$.
Call $\Phi$ a \emph{capacity} if it is nonnegative, increasing
[i.e.\ $\Phi(E)\le\Phi(F)$ whenever $E\subseteq F$],
and continuous on the right.
We assume $\Phi(\varnothing) = 0$,
although $\Phi$ may not be bounded in general.
When it is bounded, it is often assumed that
$\Phi$ is normalized
[i.e.\ $\sup_{Q\in\mathcal{K}}\Phi(Q) = 1$].

Define a difference
$\nabla_{Q_1}\Phi(Q)
{= \Phi(Q) - \Phi(Q\cup Q_1)}$,
and the successive difference
$\nabla_{Q_1,\ldots,Q_n}\Phi(Q)$
recursively by
$\nabla_{Q_n}\left(\nabla_{Q_1,\ldots,Q_{n-1}}\Phi\right)(Q)$.
A capacity $\Phi$ is called
\emph{completely alternating}
if
$\nabla_{Q_1,\ldots,Q_n}\Phi(Q) \le 0$
holds for any finite sequence $Q$, $Q_1,\ldots,Q_n$ of
$\mathcal{K}$.

The Fell topology on the space $\mathcal{F}'$ of nonempty closed
subsets in $R$ is LCH
(see Example~\ref{class.k}, or~\cite{matheron, molchanov}).
Now a natural capacity can be defined as
\begin{equation}\label{hitting}
\Phi(Q)
= \lambda(\{F\in\mathcal{F}': Q\cap F\neq\varnothing\}),
\quad
Q\in\mathcal{K}
\end{equation}
where $\lambda$ is a Radon measure on $\mathcal{F}'$.
It may be observed that
\begin{equation*}
-\nabla_{Q_1,\ldots,Q_n}\Phi(Q)
= \lambda(\{F\in\mathcal{F}':
Q\cap F =\varnothing,
Q_i\cap F \neq\varnothing,
i = 1,\ldots,n\}),
\end{equation*}
and that
$\Phi$ as defined above is always completely alternating.
Conversely, Choquet~\cite{choquet} has shown that
this property is sufficient for a given capacity to be representable
in the form above.

\begin{theorem}\label{choquet.1}
If a capacity $\Phi$ is completely alternating
then there exists a unique Radon measure $\lambda$ on $\mathcal{F}'$
satisfying (\ref{hitting}).
\end{theorem}

Theorem~\ref{choquet.1} is
referred to as ``Choquet theorem'' in \cite{berg,matheron,molchanov},
and reappears as Theorem~\ref{m.ca} below.
This result fundamentally characterizes
the distribution of a closed random set on an LCH space $R$.
That is, when $\Phi$ is normalized,
the representation (\ref{hitting})
is interpreted as the probability of a
closed random set hitting the set $Q$;
thus, $\Phi$ of Theorem~\ref{choquet.1}
is called a \emph{hitting capacity}.

Let $\mathcal{F} := \mathcal{F}'\cup\{\varnothing\}$ be
the space of closed sets which can be viewed as the one-point
compactification of $\mathcal{F}'$.
Choquet theorems over $\mathcal{F}$ and $\mathcal{F}'$
(Theorems~\ref{m.cm} and~\ref{m.ca}, respectively)
are discussed in Section~\ref{sec.cm}.
A nonnegative set function $\varphi$ on $\mathcal{K}$
is called a \emph{conjugate functional}
(or, simply a \emph{conjugate}) if it is decreasing and
continuous on the right,
and it is said to be \emph{completely monotone}
if $\nabla_{Q_1,\ldots,Q_n}\varphi(Q) \ge 0$ holds
for any $Q$ and $Q_i$'s.
A completely monotone conjugate functional $\varphi$ uniquely
corresponds to a Radon measure $\mu$ on $\mathcal{F}$
(Theorem~\ref{m.cm}) that
satisfies
\begin{equation}\label{avoidance}
\varphi(Q)
= \mu(\{F\in\mathcal{F}: Q\cap F=\varnothing\}),
\quad
Q\in\mathcal{K},
\end{equation}
in which case $\varphi$ is known as an \emph{avoidance functional}.

The conjugate $\varphi$ is bounded by $\varphi(\varnothing)$,
so we may assume $\varphi(\varnothing) = 1$ without loss of
generality.
Then the corresponding $\mu$ in (\ref{avoidance})
is a probability measure, and it has the mass
$\mu(\{\varnothing\}) = \inf_{Q\in\mathcal{K}}\varphi(Q)$.
We call $\varphi$ \emph{degenerate}
if $\inf_{Q\in\mathcal{K}}\varphi(Q) > 0$.
If an avoidance functional $\varphi$ is nondegenerate
then the measure $\mu$ of (\ref{avoidance}) is identified with
$\lambda$ of Theorem~\ref{choquet.1},
so that the hitting capacity $\Phi$ of (\ref{hitting})
is the complement of $\varphi$ of (\ref{avoidance})
[i.e., $\Phi(Q) = 1 - \varphi(Q)$].

We define another difference operator
$\Delta_{Q_1}\Phi(Q)$
$= \Phi(Q) - \Phi(Q\cap Q_1)$,
and the successive difference
$\Delta_{Q_1,\ldots,Q_n}\Phi$
recursively.
Then we call $\Phi$ \emph{completely $\cap$-monotone}
if $\Delta_{Q_1,\ldots,Q_n}\Phi(Q)\ge 0$ holds
for any finite sequence $Q$, $Q_1,\ldots,Q_n$ of $\mathcal{K}$.
Here we consider the LCH space
of $\mathcal{K}' := \mathcal{K}\setminus\{\varnothing\}$
as introduced in Example~\ref{class.k},
and form a \emph{containment} capacity
\begin{equation}\label{mono.rep}
\Phi(Q)
= \Lambda(\{E\in\mathcal{K}': E\subseteq Q\}),
\quad
Q\in\mathcal{K}
\end{equation}
using a Radon measure $\Lambda$ on $\mathcal{K}'$.
It was also shown by Choquet~\cite{choquet}
that the complete $\cap$-monotonicity can sufficiently characterize
a containment capacity.
This version of Choquet's theorem
over $\mathcal{K}'$ follows from Theorem~\ref{R-rep}.

Our investigation was inspired by the pioneering work
of Norberg~\cite{norberg} on the
existence theorems for measures on domains that extend the results of
Choquet's.
In Section~\ref{domains} we review domain theory,
and explore the notion of Scott and Lawson topology
for partially ordered sets (posets),
in which the ordering of $\mathcal{K}$ is given by
the reverse inclusion (Example~\ref{class.k}).
Consequently the conjugate $\varphi$ is increasing on the
poset $\mathcal{K}$, so $\Phi$ becomes its conjugate,
with both of them treated as continuous functions in the
Scott topology on $\mathcal{K}$, and called
\emph{Scott functions} (Definition~\ref{scott.fun}).

In Section~\ref{radon.measure}
we discuss a representation of Scott functions
in terms of a Radon measure over the
lattice $\mathcal{S}$ of nonempty Scott open sets
(over the semilattice $\mathcal{S}'$)
in Proposition~\ref{s.rep}
(Proposition~\ref{Phi.rep}, respectively);
in this paper $\mathcal{S}'$ indicates the removal of the top element from
$\mathcal{S}$.
A representation over $\mathcal{S}$ or $\mathcal{S}'$
has been suggested by Choquet~\cite{choquet},
and explored by Murofushi and Sugeno~\cite{murofushi}
in their study of monotone set functions
(Remark~\ref{km.theorem}). However, to our knowledge, this approach
has never been applied to establish Choquet's theorems in the general
case over domains.

One advantage of Radon measure approach over the technique
of~\cite{norberg}
is the lack of any separability conditions.
Throughout Sections~\ref{sec.cm}--\ref{sup.mono}
we will employ approximation theorems of
Section~\ref{approx.sec} extensively,
and establish Theorem~\ref{choquet.1} and other flavors of
Choquet theorem on domains without any appeal to separability.
Although approximation theorems are standard techniques
in harmonic analysis (cf.~\cite{berg,choquet.lecture}),
Theorems~\ref{s.approx} and~\ref{t.approx}
in this context are novel to the best of the authors' knowledge.
Let us outline the existence proof for the Choquet's theorem
over $\mathcal{F}$ (Theorem~\ref{m.cm}):
By Theorem~\ref{hofmann.mislove}, the space $\mathcal{F}$ is
homeomorphic to a compact subspace of $\mathcal{S}$.
Once $\varphi$ is approximated by Radon measures over $\mathcal{F}$
(see Definition~\ref{approx.def} and the existence proof of
Theorem~\ref{m.cm}),
Theorem~\ref{s.approx} guarantees the existence of a Radon measure
$\mu$ over $\mathcal{F}$ satisfying (\ref{avoidance}).

\subsection{Finite and locally finite random sets
and point processes}
\label{val.intro}
The continuous lattice and domain versions of Choquet's theorems
obtained in this paper provided an inspiration for the following idea
of a \emph{$k$-valuation\/} which proved useful in the study of
point-processes below.

A capacity $\Phi$ is called a \emph{valuation}
if it is completely alternating and completely $\cap$-monotone,
or equivalently if it satisfies
\begin{equation*}
  \Phi(Q_1) + \Phi(Q_2)
  = \Phi(Q_1\cup Q_2) + \Phi(Q_1\cap Q_2)
\end{equation*}
for every pair $Q_1, Q_2$ of $\mathcal{K}$.
We call $\Phi$ a \emph{$k$-valuation} 
if $\Phi$ is completely alternating and it satisfies
\begin{equation*}
  \nabla_{Q_1,\ldots,Q_{k+1}}
  \Phi\left(\bigcup_{i\neq j} Q_i\cap Q_j\right) = 0
\end{equation*}
for every $(k+1)$-tuple $Q_1,\ldots,Q_{k+1}$ of $\mathcal{K}$.
By $\mathcal{P}_k' := \{A\in\mathcal{F}': 1\le |A| \le k\}$
we denote the collection of non-empty finite subsets of
at most size~$k$.
To the best of our knowledge, the notion of a $k$-valuation and
the corresponding representation theorem below have not been studied before.

\begin{theorem}\label{val.2}
If $\Phi$ is a $k$-valuation
then the Radon measure
$\lambda$ of Theorem~\ref{choquet.1} is supported by
$\mathcal{P}_k'$.
\end{theorem}

This flavor of Choquet's theorem
is discussed in Section~\ref{val.sec}.
We should note that
the collection $\mathcal{P}_1' = \{\{a\}: a\in R\}$
is closed in $\mathcal{F}'$, and is homeomorphic to $R$.
A Radon measure on the LCH space $R$,
viewed as a capacity on $\mathcal{K}$,
is a valuation. The converse is also true,
which is a special case of Theorem~\ref{val.2} when $k=1$.

In the rest of Section~\ref{val.intro}
we assume further that the LCH space $R$ is $\sigma$-compact.
A nonnegative integer-valued random valuation $N(Q)$
of $Q\in\mathcal{K}$ on some probability space
$(\Omega,\mathcal{B},\mathbb{P})$
is called a \emph{point process}, which is said to be \emph{simple}
if $N(\{a\}) \le 1$ for every $a\in R$.
A closed subset $F$ is called \emph{locally finite}
if $F\cap Q$ is finite for every $Q\in\mathcal{K}$.
By $\mathcal{P}_{\mathrm{lf}}$ we denote
the collection of locally finite sets.

Now a simple point process $N$ is associated with
a random closed set $\xi$ taking its values
on $\mathcal{P}_{\mathrm{lf}}$, defined as $N(Q)=|\xi\cap Q|$
where $|\xi\cap Q|$ is the cardinality of the finite set $\xi\cap Q$.
Kurtz~\cite{kurtz} studied simple point processes
in terms of the avoidance functional $\varphi$,
and called $\varphi$ a \emph{zero probability function}
because it satisfies $\varphi(Q) = \mathbb{P}(N(Q)=0)$.

In order to explore representations over
$\mathcal{P}_{\mathrm{lf}}$,
he also introduced a natural extension of capacities
over the Borel $\sigma$-algebra,
and used arbitrary finite partitions of a
compact set $W\in\mathcal{K}$
by means of Borel-measurable subsets.
Since such an extension is not available for Scott
functions, we replace partitions with
opening-free partitions of subsemilattice antichain.
A finite subset $\mathcal{G}$ of $\mathcal{K}$ is called
a \emph{subsemilattice covering} of $W$
if (i) $Q_1\cap Q_2\in\mathcal{G}$
whenever $Q_1,Q_2\in\mathcal{G}$,
and (ii) $\bigcup\mathcal{G} = W$,
where the union is over all the elements of $\mathcal{G}$.
For any antichain $Q_1,\ldots,Q_k$ of
$\mathcal{G}' = \mathcal{G}\setminus\{\bigcap\mathcal{G}\}$
(i.e., $Q_i\not\subseteq Q_j$ for any pair $(Q_i, Q_j)$),
we set
\begin{equation*}
  O_{Q_1,\ldots,Q_k}
  = \bigcup\{Q\in\mathcal{G}: Q_i\not\subseteq Q,
  \mbox{ for all $i=1,\ldots,k$}\}
\end{equation*}
if the above collection for the union
contains at least one element $Q\in\mathcal{G}$;
otherwise, put $O_{Q_1,\ldots,Q_k}=\varnothing$.
We call it an \emph{opening}
if $Q_i\not\subseteq O_{Q_1,\ldots,Q_k}$ for $i=1,\ldots,k$
so that the disjoint sequence
\begin{equation*}
  Q_1\setminus O_{Q_1,\ldots,Q_k},\ldots,
  Q_k\setminus O_{Q_1,\ldots,Q_k}
\end{equation*}
partitions $W\setminus O_{Q_1,\ldots,Q_k}$.

Call an avoidance functional $\varphi$
a \emph{locally finite valuation}
if for any $\delta > 0$ and $W$ of $\mathcal{K}$
we can find a sufficiently large $n$ such that
for an arbitrary finite subsemilattice $\mathcal{G}$ covering $W$,
\begin{equation}\label{varphi.lfv}
  \varphi(W) + \sum\nabla_{Q_1,\ldots,Q_k}
  \varphi(O_{Q_1,\ldots,Q_k}) \ge 1 - \delta
\end{equation}
where the summation is over all antichains
$Q_1,\ldots,Q_k$ in $\mathcal{G}'$
such that $O_{Q_1,\ldots,Q_k}$ is an opening
and $k \le n$.

The condition (\ref{varphi.lfv}) roughly corresponds to that of
Theorem~2.13 of~\cite{kurtz} in terms of avoidance functional.
If $\varphi$ is nondegenerate
then (\ref{varphi.lfv}) is equivalently expressed in terms of
the hitting capacity $\Phi$ by
\begin{equation}
  - \sum\nabla_{Q_1,\ldots,Q_k}\Phi(O_{Q_1,\ldots,Q_k})
  \ge \Phi(W) - \delta
\end{equation}
In Section~\ref{lfv.sec} we characterize the
representation over $\mathcal{P}_{\mathrm{lf}}$
and obtain

\begin{theorem}\label{val.3}
If a hitting capacity $\Phi$ of Theorem~\ref{choquet.1}
is a locally finite valuation
then the corresponding Radon measure $\lambda$
uniquely determines the distribution of a simple point process.
\end{theorem}

For the last part of Section~\ref{val.intro}
we generally assume that $\Phi$ is unbounded.
If $\Phi$ is completely alternating
then the conjugate $\varphi(Q) = \exp[-\Phi(Q)]$
is completely monotone (Lemma~\ref{levy.lem}).
The converse is not always true even
if $\varphi$ is strictly positive.
Hence, $\Phi$ is called the \emph{L\'{e}vy exponent}
if $\Phi(Q) = -\log\varphi(Q)$ is completely alternating.
In Section~\ref{levy.exp} we demonstrate that 
the L\'{e}vy exponent has a probabilistic interpretation
analogous to that of L\'{e}vy-Khinchin formula
(cf.~\cite{berg,bertoin}),
and that
it is sufficiently characterized by the conjugate
$\varphi$ being infinitely divisible
(Proposition~\ref{inf.div}).

A point process $N$ is said to be a \emph{general Poisson process}
if there exists a Radon measure $\lambda$ on $R$
satisfying for any sequence $Q_1,\ldots,Q_k$
of disjoint compact subsets of $R$
\begin{equation}\label{pois.gen}
\mathbb{P}(N(Q_i) = n_i,\,i=1,\ldots,k)
= \prod_{i=1}^k e^{-\lambda(Q_i)}\frac{[\lambda(Q_i)]^{n_i}}{n_i!} ,
\end{equation}
where $\lambda$ is called a parameter measure; see~\cite{daley}.
It should be noted that the parameter measure $\lambda$
could be atomic with positive measure $\lambda(\{a\})$ for some
singleton~$\{a\}$,
therefore the general Poisson process $N$ may not be
simple.

\begin{theorem}\label{pois.thm}
A conjugate $\varphi$ is a zero probability function
of a general Poisson process
if and only if it is strictly positive and
satisfies
\begin{displaymath}
  \varphi(Q_1)\varphi(Q_2)
  = \varphi(Q_1\cup Q_2)\varphi(Q_1\cap Q_2)
\end{displaymath}
for every pair $Q_1,Q_2$ of $\mathcal{K}$.
\end{theorem}

We call a conjugate $\varphi$ of Theorem~\ref{pois.thm}
an \emph{exponential valuation}.
The proof of Theorem~\ref{pois.thm}
is presented at the end of Section~\ref{levy.exp}.

\subsection{Separability and the assertion of uniqueness}
\label{nonsep.intro}

The following example is the classical case
handled by Choquet's original paper~\cite{choquet},
and separability may not be an entirely natural restriction.

\begin{example}\label{ex.nonsep}
Let $R_0 = [0,1]$ be equipped with the discrete topology,
and let $R_1 = [0,1]$ be the standard Euclidean metric space.
Then the product topology $R=R_0\times R_1$ is LCH,
but not second-countable; see Example~\ref{nonsep.ex}.
Let $\pi_1$ be the canonical projection from $R$ to $R_1$,
and let $\nu$ be the standard Lebesgue measure on $R_1$.
Introduce a capacity $\Phi$
over the family $\mathcal{K}$ of compact subsets of $R$
by setting
$\Phi(Q) = \nu(\pi_1(Q))$ for $Q\in\mathcal{K}$.
This capacity is normalized, satisfies
\begin{equation*}
  \nabla_{Q_1,\ldots,Q_n}\Phi(Q)
  = \nu(\pi_1(Q)) - \nu\left(
  \pi_1(Q)\cup\bigcap_{i=1}^n\pi_1(Q_i)
  \right);
\end{equation*}
and thus, completely alternating.
Let $\mathcal{R}_1 = \{\pi_1^{-1}(\{x_1\}):x_1\in R_1\}$
be a collection of closed subsets in $R$.
Then $\mathcal{R}_1$ can be shown to be a compact subspace of
$\mathcal{F}'$ in the Fell topology, homeomorphic to $R_1$.
We can introduce a measure $\lambda$ on $\mathcal{R}_1$ by
setting
\begin{equation*}
  \lambda\left(\{\pi_1^{-1}(\{x_1\}):x_1\in B\}\right)
  = \nu(B)
\end{equation*}
for any Borel measurable subset $B$ of $R_1$,
and view it as a measure on $\mathcal{F}'$ supported by
$\mathcal{R}_1$.
One can show that $\lambda$ is the Radon measure of
Theorem~\ref{choquet.1}, establishing that
$\Phi$ represents a hitting capacity over $\mathcal{R}_1$.
\end{example}

The approach taken by Math\'{e}ron~\cite{matheron} and
Norberg~\cite{norberg} used Carath\'{e}odory's method of
construction,
and proved the unique existence of Borel measure~$\lambda$ on
$\mathcal{F}'$ satisfying~(\ref{hitting})
when the space $R$ is second-countable.
Therefore, the existence of $\lambda$ in Example~\ref{ex.nonsep}
does not follow from their versions of Choquet theorem.
Ross~\cite{ross} applied essentially the same approach
without separability, and built a measure on $\mathcal{F}$
satisfying~(\ref{hitting}) for $Q\in\mathcal{V}$ with a different
choice of space $\mathcal{V}$.
Although the conditions for the pair of
$\mathcal{V}$ and $\mathcal{F}$
are less restrictive, they require that
$\mathcal{F}$ contain all the complements $Q^c$ for
$Q\in\mathcal{V}$; thus, excluding the pair of
$\mathcal{K}$ and $\mathcal{F}$
handled in Example~\ref{ex.nonsep}.

Theorem~\ref{choquet.1} establishes the uniqueness of
$\lambda$ over Radon measures,
that is, over Borel measures which are inner regular on all open
sets and outer regular on all Borel sets
(see, e.g., Folland~\cite{folland} in the setting of LCH spaces).
If the space $R$ is second-countable,
so is $\mathcal{F}'$ (Remark~\ref{separable}).
Consequently our results include those of Matheron
and Norberg,
since Borel and Radon measures coincide for second-countable LCH
spaces.
The example below, however, demonstrates that the representation
$\lambda$ of
Theorem~\ref{choquet.1} may not be unique over Borel measures
for an $R$ that is not second-countable.

\begin{example}\label{ex.nonunq}
Let $R$ be the LCH space of Example~\ref{ex.nonsep},
and let $C_c(R)$ be the space of continuous functions on $R$ with
compact support.
If $f\in C_c(R)$ then $f(x_0,\cdot)\equiv 0$ for all but finitely
many $x_0\in R_0$.
By the Riesz theorem we can construct a Radon measure $\lambda$ on
$R$ satisfying
$\int f\,d\lambda = \sum_{x_0\in R_0}\int_0^1f(x_0,x_1)\,dx_1$
for $f\in C_c(R)$,
and define an unbounded capacity $\Phi$ by setting
$\Phi(Q) = \lambda(Q)$ for $Q\in\mathcal{K}$.
This capacity is apparently a valuation, and it has the obvious
representation $\lambda$ supported by $\mathcal{P}_1'$.
We can choose an open subset $A = R_0\times(R_1\setminus\{0\})$,
and define a Borel measure $\tilde{\lambda}$ by
setting $\tilde{\lambda}(B) = \lambda(A\cap B)$ for any
Borel-measurable subset $B$ of $R$.
Then $\tilde{\lambda}\neq\lambda$,
and $\Phi(Q) = \tilde{\lambda}(Q)$
for every $Q\in\mathcal{K}$.
Thus, $\tilde{\lambda}$ is a distinct representation,
though it is not a Radon measure;
see, e.g., Folland~\cite{folland}.
\end{example}

It may be worth pointing out that one may naively hope to demonstrate
nonuniqueness of such Borel measures by constructing
a measure on $R$ endowed with the discrete topology. Alas, no Borel measures
exist on $R$ with such a topology (at least assuming the Continuum Hypothesis (CH),
see~\cite{dudley} for a simple proof).

The measure in the example above may be made finite (and the space compact) if one is willing
to make some extra set-theoretic assumptions (such as the existence of measurable cardinals).
While the construction is straightforward, we omit it here for the sake of brevity.

It should be emphasized that the original results of
Choquet~\cite{choquet} asserted the unique existence of Radon
measure, not requiring the space $R$ to be separable.
Our treatise of Choquet theorems over domains
could lend further weight to the justification of unified Radon
measure approach
as Tjur presented so enthusiastically in his book~\cite{tjur}.

\section{Domains and topologies}
\label{domains}

Let $\mathbb{L}$ be a poset equipped with a partial order
$\le$.
A nonempty subset $E$ of $\mathbb{L}$ is called \emph{directed}
if every pair of elements in $E$ has an upper bound in $E$,
and $\mathbb{L}$ is called a \emph{directed complete poset} (or, \emph{dcpo} for
short)
if $\sup E$ exists for any directed subset $E$.
A subset $E$ is called a \emph{lower} set
if $x\in E$ and $y \le x$ imply $y \in E$,
and a lower set $E$ is called an \emph{ideal}
if it is also directed.
In particular, the ideal
$\langle x \rangle := \{z: z \le x\}$
generated by an element $x$
is called \emph{principal}.
Assuming that $\mathbb{L}$ is a dcpo,
an element $x$ of $\mathbb{L}$ is said to be ``way below'' $y$,
denoted by $x \ll y$,
if for every directed set $E \subseteq \mathbb{L}$ that satisfies
$y \le \sup E$ one can find $w \in E$ such that $x \le w$.
An element $x$ is called \emph{isolated from below}
if $x \ll x$.
A dcpo $\mathbb{L}$ is called a \emph{domain}
if (i) $\llangle x \rrangle := \{z: z \ll x\}$ is an ideal
and (ii) it satisfies $x = \sup\llangle x \rrangle$
for any $x\in \mathbb{L}$.
Every domain possesses
the \emph{strong interpolation property}:
If $x \ll z$ and $z \neq x$ then
there exists some $y \neq x$ interpolating $x \ll y \ll z$.

Throughout this paper
we frequently use \cite{scott}
as a standard reference on domain theory and generally follow the
notation introduced therein.
One notable exception is our choice of notation
$\langle\cdot\rangle$ and $\llangle\cdot\rrangle$
for generators
(which in \cite{scott} are denoted as $\downarrow\!\!\cdot$ and
$\twoheaddownarrow\cdot$, respectively).
A poset with the reversed (or ``dual'') order relation
$\le^*$ is referred to as a \emph{dual} poset, denoted by $\mathbb{L}^*$.
A subset is called \emph{filtered} in $\mathbb{L}$
if it is directed in $\mathbb{L}^*$.
An \emph{upper} set of $\mathbb{L}$
is dually defined as a lower set of $\mathbb{L}^*$,
and a filtered upper set is simply called a \emph{filter}.
In the analogous manner
we write
$\langle x \rangle^* := \{z: x \le z\}$
and
$\llangle x \rrangle^* := \{z: x \ll z\}$
(which are denoted $\uparrow\!\!{x}$ and
$\twoheaduparrow{x}$, respectively, in \cite{scott}).
For a subset $A$ we can write
$\langle A \rangle := \{z: z \le x \mbox{ for some $x\in A$ }\}$
and
$\llangle A \rrangle := \{z: z \ll x \mbox{ for some $x\in A$ }\}$.
Again, analogously we can define
$\langle A \rangle^*$ and $\llangle A \rrangle^*$.

\begin{definition}\label{topology.def}
Let $\mathbb{L}$ be a domain.
Then a subset $U$ of $\mathbb{L}$ is said to be \emph{Scott open} 
if (i) it is an upper set and (ii) $U\cap E \neq\varnothing$
holds whenever $E$ is a directed subset and
satisfies $\sup E \in U$.
A refinement can be made by introducing
additional closed upper sets $\langle x \rangle^*$,
thus defining a \emph{Lawson topology} on $\mathbb{L}$.
A subbase of the Lawson topology is formed by
all the Scott open subsets $U$
and all the lower subsets of the form
$\mathbb{L}\setminus\langle x \rangle^*$.
\end{definition}

Scott topology is $T_0$;
specifically, if $x\not\le y$ then there is a Scott open set $U$ such
that $x\in U$ and $y\not\in U$. On the other hand it is not Hausdorff
in general as the following quick example shows.
The real line $(-\infty, \infty)$
is a domain in which
$x \ll y$ is equivalent to $x < y$,
and the Scott topology consists of
open intervals $(x,\infty)$ unbounded above,
while the Lawson topology is the standard metric one.

A poset $\mathbb{L}$ is said to be a \emph{semilattice}
if $x\wedge y := \inf\{x,y\}$,
called the \emph{meet}, exists for every pair $\{x,y\}$.
Similarly we can define a \emph{sup-semilattice}
if $x \vee y := \sup\{x,y\}$, called the \emph{join},
exists for every pair $\{x,y\}$.
A poset is called a \emph{lattice}
if both the meet and the join exist for every pair,
and it is said to be a \emph{complete lattice}
if both the supremum and the infimum exist for every subset of $\mathbb{L}$.
A domain $\mathbb{L}$ is called a \emph{continuous sup-semilattice}
if it is a sup-semilattice.
It is called a \emph{continuous lattice}
if it is a complete lattice.
For example, a half-closed interval $(0, 1]$
is a continuous sup-semilattice.
It is also a lattice, but not a continuous lattice.
By $\mathrm{Scott}(\mathbb{L})$
we denote the family of Scott open subsets in a domain $\mathbb{L}$.
The poset $\mathrm{Scott}(\mathbb{L})$ ordered by inclusion
is a continuous lattice,
in which $U \ll V$ if
$U \subseteq \langle A \rangle^*$ holds
for some finite subset $A$ of $V$.

If a sup-semilattice $\mathbb{L}$ is a dcpo 
then it would suffice to check
$x = \sup\llangle x \rrangle$
in order to see whether it is a domain, 
or equivalently, to find some $z \ll x$
with $z \not\le y$ whenever $x \not\le y$.
A continuous sup-semilattice $\mathbb{L}$ is unital,
containing the top element $\hat{1} := \sup \mathbb{L}$.
If it also has the bottom element $\inf \mathbb{L}$
then it becomes a continuous lattice.
Regardless of whether
there exists a bottom element or not,
we can always form a continuous lattice,
denoted by $\check{\mathbb{L}} := \mathbb{L}\cup\{\hat{0}\}$,
by adjoining a bottom element $\hat{0}$.
Equipped with the Lawson topology,
a continuous sup-semilattice $\mathbb{L}$ is LCH,
and $\check{\mathbb{L}}$ can be viewed as the one-point compactification of
$\mathbb{L}$
(cf.\ Theorem III-1.9 of \cite{scott}).

By $\mathrm{OFilt}(\mathbb{L})$ we denote the semilattice
of Scott open filters ordered by inclusion.
Assuming that a domain $\mathbb{L}$ is a semilattice,
the way-below relation $\ll$ is said to be \emph{multiplicative}
if $a \wedge b \ll x \wedge y$ holds
whenever $a \ll x$ and $b \ll y$.
Lawson~\cite{lawson} proved that
if a domain $\mathbb{L}$ is a semilattice with multiplicative way-below relation
that has a top element $\hat{1}$ satisfying $\hat{1}\ll\hat{1}$
then $\mathrm{OFilt}(\mathbb{L})$
is a continuous lattice with the bottom element $\{\hat{1}\}$.

\begin{assumption}\label{k.ass}
In the rest of the paper we will consider
a domain $\mathbb{K}$ which is also a lattice with the top element $\hat{1}$,
and assume that the way-below relation satisfies
Lawson's conditions for the continuous lattice of open filters,
that is, (i) it is multiplicative, and (ii) $\hat{1} \ll \hat{1}$.
\end{assumption}

Now $\mathbb{K}$ is a continuous sup-semilattice
(although not necessarily a continuous lattice),
and $\mathrm{OFilt}(\mathbb{K})$ is a continuous lattice.

\begin{example}\label{class.k}
The family $\mathcal{K}$ of compact subsets of an LCH space $R$
is a domain and a lattice with reverse inclusion.
Here we have $E \ll F$ in $\mathcal{K}$
if and only if $F \subseteq \mathrm{int}(E)$
(cf.~Proposition~I-1.24.2 of \cite{scott}),
which is multiplicative.
The top element $\varnothing$ is isolated from below,
and therefore, Assumption~\ref{k.ass} holds for $\mathcal{K}$.
It should be noted
that $\mathcal{K}$ is a continuous lattice if the
entire space $R$ itself is compact,
and that each connected compact component of $R$, if any, is isolated
from below.
Let $\mathcal{F}$ denote the class of closed sets in $R$.
The Lawson topology of $\mathcal{K}$ is formed by a subbase
consisting of
$\mathcal{K}_F = \{Q\in\mathcal{K}: Q\cap F=\varnothing\}$,
$F\in\mathcal{F}$, and
$\mathcal{K}\setminus\langle{E}\rangle^*
= \{Q\in\mathcal{K}: Q\cap E^c\neq\varnothing\}$,
$E\in\mathcal{K}$.
\end{example}

The class $\mathcal{F}$ of closed subsets of $R$
is a continuous lattice with reverse inclusion,
in which $E \ll F$ if and only if there exists
$Q\in\mathcal{K}$ such that $E\cup Q = R$ and $F\cap Q =
\varnothing$
(cf. Section~III-1 of~\cite{scott}).
The Lawson topology on $\mathcal{F}$ is also known as the Fell topology
(see~\cite{scott,matheron}),
with a subbase consisting of
$\mathcal{F}_Q = \{F\in\mathcal{F}: F\cap Q=\varnothing\}$,
$Q\in\mathcal{K}$,
and
$\mathcal{F}\setminus\langle{E}\rangle^*
= \{F\in\mathcal{F}: F\cap E^c\neq\varnothing\}$,
$E\in\mathcal{F}$.
The next result is a version of Hofmann-Mislove
theorem; see~\cite{scott}.

\begin{theorem}\label{hofmann.mislove}
The map $\Psi(F) = \mathcal{K}_F$
is a homeomorphism from $F\in\mathcal{F}$ to
$\Psi(F)\in\mathrm{OFilt}(\mathcal{K})$.
\end{theorem}

\begin{proof}
The map $\Psi$ is clearly injective.
Let $\mathcal{V}\in\mathrm{OFilt}(\mathcal{K})$ be fixed arbitrarily,
and let $F = R\setminus\cup_{Q\in\mathcal{V}} Q$.
Since $\cup_{Q\in\mathcal{V}} Q = \cup_{Q\in\mathcal{V}}\mathrm{int}(Q)$,
we find $F\in\mathcal{F}$.
If $Q\in\mathcal{V}$ then $Q\cap F=\varnothing$.
If $E\in\Psi(F)$ then
$E\subseteq\cup_{Q\in\mathcal{V}}\mathrm{int}(Q)$,
therefore $E \subseteq Q$ for some $Q\in\mathcal{V}$.
Thus we obtain $\mathcal{V} = \Psi(F)$,
and consequently, $\Psi$ is bijective.

Moreover,
$\Psi(\mathcal{F}_Q) = \{\mathcal{V}\in\mathrm{OFilt}(\mathcal{K}):
Q\in\mathcal{V}\}$,
$Q\in\mathcal{K}$,
and
$\Psi(\mathcal{F}\setminus\langle{E}\rangle^*)
= \mathrm{OFilt}(\mathcal{K})\setminus
\{\mathcal{V}\in\mathrm{OFilt}(\mathcal{K}):
\Psi(E)\subseteq\mathcal{V}\}$,
$E\in\mathcal{F}$,
form a subbase for $\mathrm{OFilt}(\mathcal{K})$,
which implies that $\Psi$ is a homeomorphism.
\end{proof}

\begin{remark}
The top element $\varnothing$ of the continuous lattice $\mathcal{F}$ is
not isolated from below unless $R$ is compact,
and therefore, the domain $\mathcal{F}$ does not satisfy
Assumption~\ref{k.ass}.
\end{remark}

\begin{definition}\label{scott.def}
By $\mathcal{S}$
we denote the lattice
$\mathrm{Scott}(\mathbb{K})\setminus\{\varnothing\}$
of nonempty Scott open subsets in $\mathbb{K}$
ordered by inclusion.
We view $\mathrm{Scott}(\mathbb{K})$ as
an extension of $\mathcal{S}$
by adjoining the bottom element $\varnothing$,
and denote it by $\check{\mathcal{S}}$.
By $\mathcal{F}$
we denote the lattice
$\mathrm{OFilt}(\mathbb{K})$ of Scott open filters;
there should be no confusion with
the class $\mathcal{F}$ of closed sets
in light of Theorem~\ref{hofmann.mislove}.
\end{definition}

The lattice $\mathcal{S}$
itself becomes a continuous lattice
with the bottom element $\{\hat{1}\}$,
and it becomes a compact Hausdorff space when equipped with the
Lawson topology.
For $x\in \mathbb{K}$ we can define a filter
$\mathcal{S}_x := \{U\in\mathcal{S}: x \in U\}$.
Given a directed subset $\mathcal{E}$ of $\mathcal{S}$ satisfying
$\sup\mathcal{E} \in \mathcal{S}_x$,
we can find some $U \in \mathcal{E}$ which contains an $x$
such that $U \in \mathcal{S}_x$;
thus, $\mathcal{S}_x$ is Scott-open.
The collection of $\mathcal{S}_x$, $x\in \mathbb{K}$,
becomes an open subbase for the Scott topology of
$\mathcal{S}$;
in fact, it forms a base since $\mathbb{K}$ is sup-semilattice.

We can view $\mathcal{F}$ as a base
for the Scott topology of $\mathbb{K}$,
therefore we can express
a principal filter $\langle U \rangle^*_{\mathcal{S}}
:= \{W\in\mathcal{S}: U \subseteq W\}$
on $\mathcal{S}$
as the intersection of
$\langle V \rangle^*_{\mathcal{S}}$,
$V\in\mathcal{F}$ with $V \subseteq U$.
Thus, the Lawson topology of $\mathcal{S}$
is formed by a subbase consisting of
$\mathcal{S}_x$, $x\in \mathbb{K}$, and
$\mathcal{S}\setminus\langle V \rangle^*_{\mathcal{S}}$,
$V\in\mathcal{F}$.

\begin{lemma}\label{f.closed}
$\mathcal{F}$ is a closed subset of the compact Hausdorff space
$\mathcal{S}$.
\end{lemma}

\begin{proof}
A subbase of the Lawson topology of $\mathcal{F}$ consists of
$\mathcal{F}_x = \{V\in\mathcal{F}:x\in V\}$, $x\in \mathbb{K}$,
and
$\mathcal{F}\setminus\langle V \rangle^*_{\mathcal{F}}$,
$V\in\mathcal{F}$,
which coincides with the topology induced
by the Lawson topology of $\mathcal{S}$.
Since $\mathcal{F}$ is a continuous lattice,
$\mathcal{F}$ is compact in the Lawson topology,
and therefore, it is closed in $\mathcal{S}$.
\end{proof}

\section{Scott functions and their representations}
\label{radon.measure}

A map $f$ from a domain $\mathbb{L}$ to another domain $\mathcal{L}$
is called \emph{Scott-continuous}
if it is continuous in their respective Scott topologies.
% * <motoyamachida@gmail.com> 00:00:00 05 Nov 2017 UTC-0500:
% Lemma on Lawson-continuity is no longer needed and removed.
Recall that the domains of interest satisfy Assumption~\ref{k.ass},
and that their Lawson topologies are LCH.
Thus, we call $f$ simply ``continuous''
(not \emph{Lawson-continuous})
if $f$ is continuous with respect to the Lawson topology.
Scott-continuity implies monotonicity, and the
following equivalent conditions hold
(Proposition II-2.1 of \cite{scott}).

\begin{proposition}\label{scott}
Suppose $f$ maps some domain $\mathbb{L}$ to another domain $\mathcal{L}$.
Then the following statements are equivalent:
\begin{enumerate}
\item
$f$ is Scott-continuous;
\item
$f(x) = \sup f(E)$
whenever $E$ is a directed subset of $\mathbb{L}$
converging to $x$;  
\item
$f(x) = \sup f(\llangle x\rrangle)$ for every $x\in \mathbb{L}$.
\end{enumerate}
\end{proposition}

When a real-valued function $f$ on a domain $\mathbb{L}$
is considered, $f$ is Scott-continuous
if and only if $f$ is increasing
[i.e.\ $f(x)\le f(y)$ whenever $x\le y$]
and lower semi-continuous (l.s.c)
on $\mathbb{L}$ equipped with the Scott topology.

\begin{definition}\label{scott.fun}
Let $\varphi$ and $\Phi$ be nonnegative functions on $\mathbb{K}$.
Then we call $\varphi$ a \emph{Scott function}
if $\varphi$ is Scott-continuous from $\mathbb{K}$ to
$[0,\infty)$,
and call $\Phi$ a \emph{conjugate Scott function}
(or, simply a \emph{conjugate})
if $\Phi$ is Scott-continuous from $\mathbb{K}$
% * <motoyamachida@gmail.com> 00:00:00 03 Dec 2017 UTC-0500:
% The exact range is $(-\infty,-\Phi(\hat{1})]$.
to $[0,\infty)^*$
(i.e., to the dual poset of $[0,\infty)$).
\end{definition}

A Scott function $\varphi$ is increasing
and bounded, while
a conjugate $\Phi$ is decreasing, and not necessarily bounded.
Without loss of generality
we set $\varphi(\hat{1}) = 1$ and $\Phi(\hat{1}) = 0$
throughout this paper.
Observe that
an increasing (or a decreasing) nonnegative function
$\psi$ is a Scott function (a conjugate Scott function, respectively)
if for any $x\in \mathbb{K}$ and
$\varepsilon > 0$ there exists some element $z \ll x$ such that
$|\psi(x)-\psi(y)| < \varepsilon$ whenever $z \ll y \le x$.
In Section~\ref{choquet.intro}
a conjugate functional $\varphi$ over $\mathcal{K}$
corresponds to a Scott function on the domain $\mathcal{K}$ of
Example~\ref{class.k},
and a capacity $\Phi$ over $\mathcal{K}$
to a conjugate Scott function on $\mathcal{K}$.

% * <motoyamachida@gmail.com> 00:00:00 14 Nov 2017 UTC-0500:
% \Lambda is not appropriate for a measure, and replaced.
Let $\xi$ be a Scott-continuous map from $(0,1]$
to the continuous lattice $\mathcal{S}$.
Then $\xi$ is a Borel-measurable map from $(0,1]$
to the compact Hausdorff space $\mathcal{S}$.
By the Riesz theorem there exists a unique Radon measure $\mu$
on $\mathcal{S}$
which corresponds to the positive functional
\begin{equation}\label{pos.func}
I(f) = \int_0^1 f(\xi(r)) dr = \int f d\mu
\end{equation}
over the space $C(\mathcal{S})$
of continuous functions on $\mathcal{S}$.

\begin{lemma}\label{nu.cons}
Let $\nu$ be the standard Lebesgue measure on $(0,1]$,
and let $\mathcal{U}$ be a Scott open subset of $\mathcal{S}$.
Then the Radon measure $\mu$ of (\ref{pos.func}) satisfies
$\mu(\mathcal{U}) = \nu(\xi^{-1}(\mathcal{U}))$.
\end{lemma}

\begin{proof}
By the Riesz representation we have
\begin{equation*}
  \mu(\mathcal{U})
  = \sup\{I(f): f\in C(\mathcal{S}), 0\le f\le 1,
  \mathrm{supp}f\subseteq \mathcal{U}\} ,
\end{equation*}
where $\mathrm{supp} f$ denotes the support of $f$.
Thus, we find $\mu(\mathcal{U}) \le \nu(\xi^{-1}(\mathcal{U}))$.
Assuming $\xi^{-1}(\mathcal{U}) \neq \varnothing$,
we can write $\xi^{-1}(\mathcal{U}) = (r,1]$ for some $0 \le r < 1$.
By choosing an arbitrary $0 < \varepsilon < 1-r$,
we can construct an $f\in C(\mathcal{S})$ such that $f \equiv 1$ on 
the compact subset $\langle \xi(r+\varepsilon) \rangle^*$ of $\mathcal{U}$
and $\mathrm{supp} f \subseteq \mathcal{U}$
using a locally compact version of Urysohn's lemma
(cf. Section 4.5 of \cite{folland}).
By the Riesz representation we obtain
$\mu(\mathcal{U}) \ge 1 - r - \varepsilon$,
which completes the proof.
\end{proof}

Let $\varphi$ be a Scott function.
Then we can construct a Scott-continuous map $\xi(r) = \{z\in \mathbb{K}: \varphi(z) > 1-r\}$
from $(0,1]$ to $\mathcal{S}$,
where the Scott-continuity of $\xi$ is implied by
$\xi^{-1}(\mathcal{S}_x) = (1-\varphi(x),1]$.
By Lemma~\ref{nu.cons}
the corresponding Radon measure $\mu$ of (\ref{pos.func})
on $\mathcal{S}$ satisfies
$\mu(\mathcal{S}_x) = \nu(\xi^{-1}(\mathcal{S}_x)) = \varphi(x)$,
which shows the existence of $\mu$ for the following proposition.

\begin{proposition}\label{s.rep}
There exists a Radon measure $\mu$ on $\mathcal{S}$
that satisfies $\varphi(x) = \mu(\mathcal{S}_x)$
if and only if $\varphi$ is a Scott function.
\end{proposition}

\begin{proof}
It suffices to show that $\varphi(x) = \mu(\mathcal{S}_x)$
is Scott-continuous.
Suppose that $E$ is a directed subset of $\mathbb{K}$
converging to $x = \sup E$.
Then for any $U \in\mathcal{S}$
satisfying $\delta_x(U) = 1$
there exists some $z\in E$ such that
$\delta_z(U) = 1$.
Thus, we can view
$\{\delta_z\}_{z\in E}$ as an increasing net of l.s.c.~functions
on $\mathcal{S}$ converging to
$\delta_x = \sup_{z\in E} \delta_z$.
By applying the monotone convergence theorem for nets
(MCT for short; Proposition 7.12 of~\cite{folland}),
we obtain
\begin{equation*}
  \sup_{z\in E} \varphi(z)
  = \sup_{z\in E} \int\delta_z d\mu
  = \int\delta_x d\mu
  = \varphi(x),
\end{equation*}
which implies that $\varphi$ is Scott-continuous.
\end{proof}

\begin{remark}\label{km.theorem}
Proposition~\ref{s.rep} provides a representation
of a Scott function on the compact Hausdorff space $\mathcal{S}$.  
In the context of Krein-Milman theorem
the collection $\mathcal{C}_1$ of
Scott functions can be viewed as a compact convex space.
For each $U\in\mathcal{S}$
the indicator function $I_U$
[i.e., $I_U(x) = 1$ if $x\in U$; otherwise, $I_U(x) = 0$]
is a Scott function.
Moreover, $\mathcal{S}$ is naturally embedded onto $\mathcal{C}_1$
as the collection $\mathrm{ex}(\mathcal{C}_1)$ of
extreme points in $\mathcal{C}_1$.
Choquet~\cite{choquet}
and others (e.g.\ \cite{berg}) showed that
$\mathcal{C}_1$ is the closure of convex hull of
$\mathrm{ex}(\mathcal{C}_1)$, therefore an element
$\varphi\in\mathcal{C}_1$ is represented by the integral
$\ell(\varphi) =
\int_{\mathrm{ex}(\mathcal{C}_1)}\ell(\rho)d\mu(\rho)$
for any continuous linear functional $\ell$ on $\mathcal{C}_1$.
In particular, the measure $\mu$ on $\mathcal{S}$
satisfies
$\varphi(x) = \mu(\mathcal{S}_x)$, $x\in \mathbb{K}$.
In contrast the construction of $\mu$ in
Proposition~\ref{s.rep}
is elementary, and was first presented
by Murofushi and Sugeno in~\cite{murofushi}.
\end{remark}

\begin{lemma}\label{s.cover}
Let $\langle\llangle x\rrangle^*\rangle^*_{\mathcal{S}}
= \{V\in\mathcal{S}: \llangle x\rrangle^*\subseteq V\}$
be the closed upper subset of $\mathcal{S}$
generated by $\llangle x\rrangle^*\in\mathcal{S}$,
and let $\mathcal{S}' = \mathcal{S}\setminus\{\mathbb{K}\}$.
Then
\begin{math}
  \mathcal{S}\setminus
  \langle\llangle x\rrangle^*\rangle^*_{\mathcal{S}},
\end{math}
$x\in \mathbb{K}$, cover the LCH space $\mathcal{S}'$.
\end{lemma}

\begin{proof}
Let $\mathbb{K}^*$ be the dual poset of $\mathbb{K}$.
Then the net of the closed sets
$\langle\llangle x\rrangle^*\rangle^*_{\mathcal{S}}$
indexed by $x\in \mathbb{K}^*$
converges to the singleton $\{\mathbb{K}\}$.
Thus, their complements cover $\mathcal{S}'$.
\end{proof}

Let $\mu$ be the Radon measure of Proposition~\ref{s.rep}.
In the proof of Lemma~\ref{s.cover}
for any $x\in \mathbb{K}$
we can find some $y\ll x$ so that
$\langle\llangle y\rrangle^*\rangle^*_{\mathcal{S}}
\subseteq \mathcal{S}_x \subseteq
\langle\llangle x\rrangle^*\rangle^*_{\mathcal{S}}$,
therefore we obtain
\begin{math}
  \inf_{x\in \mathbb{K}}
  \mu\left(\langle\llangle x\rrangle^*\rangle^*_{\mathcal{S}}\right)
  = \inf_{x\in \mathbb{K}}
  \mu\left(\mathcal{S}_x\right).
\end{math}
By MCT we have
\begin{equation*}
  \inf_{x\in \mathbb{K}} \varphi(x)
  = \inf_{x\in \mathbb{K}}
  \mu\left(\langle\llangle x\rrangle^*\rangle^*_{\mathcal{S}}\right)
  = \mu(\{\mathbb{K}\});
\end{equation*}
thus, $\mu$ is non-atomic at $\mathbb{K}$ whenever $\inf_{x\in \mathbb{K}} \varphi(x) = 0$.

\begin{definition}\label{degen.def}
We call a Scott function $\varphi$ \emph{degenerate}
if $\inf_{x\in \mathbb{K}} \varphi(x) > 0$.
When $\varphi$ is nondegenerate,
the representation $\mu$
of Proposition~\ref{s.rep}
can be viewed as a Radon measure on
$\mathcal{S}'$.
\end{definition}

Let $\Phi$ be a conjugate Scott function,
and let $\mathcal{S}^x = \{U\in\mathcal{S}:x\not\in U\}$
be the complement of $\mathcal{S}_x$.
Observe that $\mathcal{S}_x$ is an open neighborhood of $\mathbb{K}$
in $\mathcal{S}$,
and therefore, that $\mathcal{S}^x$
is a compact subset of $\mathcal{S}'$.
If $\sup_{z\in \mathbb{K}}\Phi(z) < \infty$,
we may normalize $\Phi$ [i.e.\ require that $\sup_{z\in \mathbb{K}}\Phi(z) = 1$].
Then $\varphi(z) = 1 - \Phi(z)$ is a nondegenerate Scott
function,
and its representation $\mu$ of Proposition~\ref{s.rep}
is a Radon measure on $\mathcal{S}'$
satisfying $\Phi(x) = \mu(\mathcal{S}^x)$.
Suppose that $\sup_{z\in \mathbb{K}}\Phi(z) = \infty$.
Then the map $\xi(r) = \{z\in \mathbb{K}: \Phi(z) < r\}$
is Scott-continuous from $(0,\infty)$ to $\mathcal{S}'$.
By the Riesz theorem there is a Radon measure $\lambda$ on
$\mathcal{S}'$ satisfying
\begin{equation*}
  \int_0^{\infty} f(\xi(r)) dr = \int f d\lambda,
  \quad f \in C_c(\mathcal{S}'),
\end{equation*}
where $C_c(\mathcal{S}')$ denotes
the space of continuous functions on $\mathcal{S}'$
with compact support.
Similar to the proof of
Lemma~\ref{nu.cons} and Proposition~\ref{s.rep},
we can show that
$\lambda(\mathcal{S}^x)
= \nu(\xi^{-1}(\mathcal{S}^x))
= \nu([0,\Phi(x)]) = \Phi(x)$,
and that
$-\Phi(x) = -\lambda(\mathcal{S}^x)$
is Scott-continuous;
thus, establishing

\begin{proposition}\label{Phi.rep}
There exists a Radon measure $\lambda$ on $\mathcal{S}'$
that satisfies $\Phi(x) = \lambda(\mathcal{S}^x)$
if and only if
$\Phi$ is a conjugate Scott function on $\mathbb{K}$.
\end{proposition}

\section{Approximation theorems}
\label{approx.sec}

As introduced in Section~\ref{domains}, the
Lawson topology of $\mathcal{S}$ is compact Hausdorff,
and therefore, normal.
Let $C(\mathcal{S})$ be the space of continuous functions on
$\mathcal{S}$,
and let $I_{\mathcal{S}_x}(U)$ be the indicator function of a
Scott-open subset $\mathcal{S}_x$ over $U\in\mathcal{S}$.
By applying Urysohn's lemma we obtain the following result.

\begin{lemma}\label{urysohn}
If $a \ll x$ then there exists $f\in C(\mathcal{S})$
such that
\begin{math}
  I_{\mathcal{S}_a}(U)\le f(U)\le I_{\mathcal{S}_x}(U)
\end{math}
for every $U\in\mathcal{S}$.
\end{lemma}

\begin{proof}
Observe that the closed set
$\langle\llangle a\rrangle^*\rangle^*_{\mathcal{S}}$
satisfies
\begin{math}
  \mathcal{S}_a\subseteq
  \langle\llangle a\rrangle^*\rangle^*_{\mathcal{S}}
  \subseteq\mathcal{S}_x .
\end{math}
By Urysohn's lemma there exists
$f\in C(\mathcal{S})$
such that
$0\le f\le 1$,
$f(U)=0$ for $U\not\in\mathcal{S}_x$,
and $f(U)=1$ for
$U\in\langle\llangle a\rrangle^*\rangle^*_{\mathcal{S}}$,
as desired.
\end{proof}

We consider the collection $M^+(\mathcal{S})$
of Radon measures on $\mathcal{S}$,
and equip it with the weak$*$ topology,
in which the convergence $\mu = \lim_\alpha \mu_\alpha$ of Radon
measures is characterized by
$\int f d\mu = \lim_\alpha \int f d\mu_\alpha$
for every $f \in C(\mathcal{S})$.
It should be noted that
$M^+(\mathcal{S})$ is complete,
that is,
that every Cauchy net $\{\mu_\alpha\}$ has the limit $\mu$
in $M^+(\mathcal{S})$ (cf. Section~12 of~\cite{choquet.lecture}),
and that
\begin{math}
M^+_1(\mathcal{S})
= \{\mu\in M^+(\mathcal{S}): \mu(\mathcal{S}) = 1\}
\end{math}
is compact
(cf.\ Corollary~12.7 of \cite{choquet.lecture}).

\begin{definition}\label{approx.def}
Let $\varphi$ be a Scott function on $\mathbb{K}$,
and let $\mathcal{H}$ be a closed subset of $\mathcal{S}$.
A net $\{\mu_F\}$ of $M^+(\mathcal{S})$
indexed by finite subsets $F$ of $\mathbb{K}$ is said to
\emph{approximate} $\varphi$ over $\mathcal{H}$
if each $\mu_F$ is supported by $\mathcal{H}$
[i.e.\ $\mu_F(\mathcal{S}\setminus\mathcal{H}) = 0$]
and satisfies
\begin{math}
  \varphi(x) = \mu_F(\mathcal{S}_x)
\end{math}
for all $x\in F$.

A Radon measure $\mu$ on $\mathcal{S}$ is said to
\emph{represent} $\varphi$ over $\mathcal{H}$
if $\mu$ is supported by $\mathcal{H}$
and satisfies
$\varphi(x) = \mu(\mathcal{H}_x)$, $x\in \mathbb{K}$,
where we simply write
$\mathcal{H}_x = \mathcal{H}\cap\mathcal{S}_x$.
\end{definition}

\begin{theorem}\label{s.approx}
If a Scott function $\varphi$ is approximated over $\mathcal{H}$
then there exists some $\mu\in M^+_1(\mathcal{S})$ that
represents $\varphi$ over $\mathcal{H}$.
\end{theorem}

To prove Theorem~\ref{s.approx} we first need the following lemma.

\begin{lemma}\label{portmanteau}
Let $\{\mu_\alpha\}$ be a net
converging to $\mu$ in $M^+(\mathcal{S})$.
Then $\mu(\mathcal{U})\le\liminf_\alpha\mu_\alpha(\mathcal{U})$
for any open subset $\mathcal{U}$ of $\mathcal{S}$.
\end{lemma}

\begin{proof}
Let $\mathcal{U}$ be an open subset of $\mathcal{S}$,
and let $\varepsilon > 0$ be arbitrary.
Since $\mu$ is a Radon measure,
there exists a compact subset
$\mathcal{V}\subseteq\mathcal{U}$ such that
$\mu(\mathcal{U})-\varepsilon<\mu(\mathcal{V})$.
By Urysohn's lemma we can find $f\in C(\mathcal{S})$
such that $I_{\mathcal{V}}\le f\le I_{\mathcal{U}}$,
and obtain
\begin{equation*}
  \mu(\mathcal{U})-\varepsilon
  <\mu(\mathcal{V})
  \le \int f d\mu
  = \lim_{\alpha} \int f d\mu_\alpha
  \le\liminf_{\alpha} \mu_\alpha(\mathcal{U})
\end{equation*}
which completes the proof.
\end{proof}

\begin{proof}[Proof of Theorem~\ref{s.approx}]
Let $\{\mu_F\}$ be an approximating net of $M^+_1(\mathcal{S})$.
Without loss of generality assume $\hat{1}\in F$.
Since
\begin{math}
  \mu_F(\mathcal{S}) = \mu_F(\mathcal{S}_{\hat{1}}) =
  \varphi(\hat{1}) = 1,
\end{math}
the net $\{\mu_F\}$ is a subset of $M^+_1(\mathcal{S})$,
therefore it has a subnet $\{\mu_{F'}\}$
converging to some $\mu\in M^+_1(\mathcal{S})$.
Let $x\in \mathbb{K}$ and fix an arbitrary $\varepsilon > 0$.
By Lemma~\ref{portmanteau}
we can observe that
\begin{math}
  \mu(\mathcal{S}_x)
  \le \liminf_{F'} \mu_{F'}(\mathcal{S}_x) \le \varphi(x).
\end{math}
By the Scott-continuity of $\varphi$
we can find some $a\ll x$ satisfying
$\varphi(x)-\varepsilon < \varphi(a)$.
By Lemma~\ref{urysohn}
there exists $f\in C(\mathcal{S})$
such that
$I_{\mathcal{S}_a}\le f\le I_{\mathcal{S}_x}$.
Together we obtain
\begin{equation*}
\varphi(x) - \varepsilon
< \varphi(a) \le \limsup_{F'} \mu_{F'}(\mathcal{S}_a)
\le \lim_{F'} \int f d\mu_{F'}
= \int f d\mu
\le \mu(\mathcal{S}_x),
\end{equation*}
which implies that
$\varphi(x) = \mu(\mathcal{S}_x)$.
Again by Lemma~\ref{portmanteau}
we have
\begin{math}
  0\le\mu(\mathcal{S}\setminus\mathcal{H})
  \le\liminf_{F'} \mu_{F'}(\mathcal{S}\setminus\mathcal{H})
  \le 0;
\end{math}
thus, $\mu$ is also supported by $\mathcal{H}$.
\end{proof}

We set $\mathcal{S}' = \mathcal{S}\setminus\{\mathbb{K}\}$,
and view it as an LCH space.
Then $\mathcal{S}^x$ is a compact subset of $\mathcal{S}'$.

\begin{definition}
Let $\Phi$ be a conjugate Scott function on $\mathbb{K}$,
and let $\mathcal{H}' = \mathcal{H}\setminus\{\mathbb{K}\}$
be a closed subset of the LCH space $\mathcal{S}'$.
A Radon measure $\lambda$ on $\mathcal{S}'$ is said to
\emph{represent} $\Phi$ over $\mathcal{H}'$
if $\lambda$ is supported by $\mathcal{H}'$
and satisfies $\Phi(x) = \lambda(\mathcal{H}^x)$,
$x\in \mathbb{K}$,
where we customarily write
$\mathcal{H}^x = \mathcal{H}'\cap\mathcal{S}^x$.
When $\lambda(\mathcal{S}')<\infty$,
we identify a measure $\lambda$ on $\mathcal{S}'$
interchangeably as a measure on $\mathcal{S}$
with $\lambda(\{\mathbb{K}\})=0$.
\end{definition}

\begin{theorem}\label{t.approx}
Let $\{\lambda_\alpha\}$ be a net of Radon measures on
$\mathcal{S}'$ supported by $\mathcal{H}'$.
Suppose that
$\lambda_\alpha(\mathcal{S}^x) \le \Phi(x)$
for any $x\in \mathbb{K}$,
and that
for each $x\in \mathbb{K}$ there is some $\beta$
such that $\lambda_\alpha(\mathcal{H}^x) = \Phi(x)$
for all $\alpha\succ\beta$.
Then there exists some $\lambda\in M^+(\mathcal{S}')$
that represents $\Phi$ over $\mathcal{H}'$.
\end{theorem}

\begin{proof}
Let $\mathcal{V}$ be a compact subset of $\mathcal{H}'$.
By Lemma~\ref{s.cover} we can find a finite sequence
$x_1,\ldots,x_n\in \mathbb{K}$ and some $z\le x_i$ for all $i$
so that
\begin{math}
  \mathcal{V}\subseteq
  \bigcup_{i=1}^n \mathcal{S}\setminus
  \langle\llangle x_i\rrangle^*\rangle^*_{\mathcal{S}}
  \subseteq\mathcal{S}^z
\end{math}
Thus, the net satisfies
\begin{math}
  \lambda_\alpha(\mathcal{V})
  \le \lambda_\alpha(\mathcal{S}^z)
  \le \Phi(z)
\end{math}
for each $\alpha$,
therefore it is relatively compact in
$M^+(\mathcal{S}')$ (see~Section~12 of~\cite{choquet.lecture}).
Let $\{\lambda_{\alpha'}\}$ be a converging subnet,
$\lambda$ be the limit of the subnet,
and let $\varepsilon > 0$ be arbitrary.
Since $\Phi(x)$ and $\lambda(\mathcal{S}^x)$
are conjugate Scott functions,
we can find some $a\ll x$ satisfying
$\Phi(a)<\Phi(x)+\varepsilon$ and
$\lambda(\mathcal{S}^a)<\lambda(\mathcal{S}^x)+\varepsilon$.
Similarly to Lemma~\ref{urysohn} (but applying
the locally compact version of Urysohn's lemma)
we can find $g\in C_c(\mathcal{S}')$ such that
$I_{\mathcal{S}^x}\le g\le I_{\mathcal{S}^a}$.
Thus, we obtain
\begin{equation*}
  \lambda(\mathcal{S}^x)
  \le \int g d\lambda
  = \lim_{\alpha'}\int g d\lambda_{\alpha'}
  \le \liminf_{\alpha'}\lambda_{\alpha'}(\mathcal{S}^a)
  \le \Phi(a) < \Phi(x)+\varepsilon ,
\end{equation*}
and
\begin{equation*}
  \Phi(x)
  \le\limsup_{\alpha'}\lambda_{\alpha'}(\mathcal{S}^x)
  \le\lim_{\alpha'}\int g d\lambda_{\alpha'}
  = \int g d\lambda
  \le\lambda(\mathcal{S}^a)
  < \lambda(\mathcal{S}^x)+\varepsilon;
\end{equation*}
thus, the equality holds.
Similarly to the proof of Theorem~\ref{s.approx}
we can observe that Lemma~\ref{portmanteau} holds for
a Radon measure $\lambda$ on $\mathcal{S}'$
thus
\begin{math}
  0\le\lambda(\mathcal{S}'\setminus\mathcal{H}')
  \le\liminf_{\alpha'}
  \lambda_{\alpha'}(\mathcal{S}'\setminus\mathcal{H}')
  =0 .
\end{math}
\end{proof}

\begin{remark}
Theorem~\ref{t.approx}
provides an alternative construction of Radon measure
of Proposition~\ref{Phi.rep}
when $\sup_{x\in \mathbb{K}}\Phi(x) = \infty$.
For each $a\in \mathbb{K}^*$
one can define the Scott function
$\psi_a(x) = \max(\Phi(a)-\Phi(x),0)$,
which is nondegenerate with
$\psi_a(\hat{1}) = \Phi(a)$.
By Proposition~\ref{s.rep}
we have a representation $\lambda_a$ on $\mathcal{S}'$ for $\psi_a$.
Observe that
$\lambda_a(\mathcal{S}^x)
= \Phi(a) - \lambda_a(\mathcal{S}_x)
\le \Phi(x)$,
so the equality holds if $\Phi(a) \ge \Phi(x)$.
Hence, we can apply Theorem~\ref{t.approx},
and show the existence of a Radon measure $\lambda$
on $\mathcal{S}'$ satisfying
$\Phi(x) = \lambda(\mathcal{S}^x)$.
\end{remark}

\section{Choquet theorems on domains}
\label{sec.cm}

Let $F$ be a semilattice,
and let $\phi$ be a function on $F$.
Then we can define a difference operator $\nabla_{z}$ by
$\nabla_{z} \phi(x) = \phi(x) - {\phi(x\wedge z)}$,
and the \emph{successive difference operator}
$\nabla_{z_1,\ldots,z_n}$ recursively by
$\nabla_{z_1,\ldots,z_n} \phi
= \nabla_{z_n}(\nabla_{z_1,\ldots,z_{n-1}} \phi)$
for $n = 2,3,\ldots$.
The operator $\nabla_{z_1,\ldots,z_n}$ does not depend on an order of
$z_i$'s, nor a repetition of elements,
and therefore, it is denoted by $\nabla_A$ for a finite subset
$A = \{z_1,\ldots,z_n\}$.

\begin{definition}
An increasing function $\phi$ is called \emph{completely monotone}
if $\nabla_{A} \phi \ge 0$ holds
for every nonempty finite subset $A$.
We can call similarly a decreasing function $\phi$ \emph{completely alternating}
if $\nabla_{A} \phi \le 0$ for each $A$.
\end{definition}

\begin{proposition}\label{cm.r}
Suppose $m$ is a finite measure on $F$,
and $\langle x \rangle$ is measurable for each $x\in F$.
Then
\begin{equation}\label{cm.phi}
  \phi(x) = m(\langle x \rangle)
\end{equation}
is nonnegative and completely monotone.
\end{proposition}

\begin{proof}
For any nonempty finite subset $A$
we can express
\begin{equation}\label{nabla.phi}
\nabla_A \phi(x) =
  \sum_{B \subseteq A} (-1)^{|B|}
  \phi(\bigwedge B\wedge x),
\end{equation}
where
$\bigwedge B\wedge x$
denotes the greatest lower bound of $B\cup\{x\}$.
By applying the inclusion-exclusion principle we can also show that
\begin{align*}
m\left(\bigcup_{z\in A}\langle z \rangle \cap \langle x \rangle \right)
& = \sum_{B\subseteq A, B\neq\varnothing}
(-1)^{|B|+1}
m\left(\bigcap_{z\in B}\langle z \rangle \cap \langle x \rangle \right)
\\ &
= \sum_{B\subseteq A, B\neq\varnothing}
(-1)^{|B|+1} \phi(\bigwedge B \wedge x)
\end{align*}
Comparing the sums above, we obtain
\begin{equation}\label{nabla.r}
\nabla_A \phi(x) = m(\langle x \rangle)
- m\left(\langle A \rangle \cap \langle x \rangle \right)
= m\left(\langle x \rangle \setminus \langle A \rangle \right),
\end{equation}
which immediately implies $\nabla_A \phi \ge 0$.
\end{proof}

The converse of Proposition~\ref{cm.r} is also true
if $F$ is finite.
We say ``$x$ covers $z$'' in $F$
if $z < x$ and there is no other element of $F$ between $z$ and $x$.
We set $r(x) = \phi(x)$
for the bottom element $x = \bigwedge F$,
and $r(x) = \nabla_A\phi(x)$ with
where $A$ is the collection of all the elements covered by $x$
when $x > \bigwedge F$.
Then we can construct a measure
$m(E) = \sum_{x\in E} r(x)$ for each $E\subseteq F$.
% * <motoyamachida@gmail.com> 00:00:00 09 Nov 2017 UTC-0500:
% I replaced ``It is a straightforward exercise to show
% $\phi(x) = m(\langle x \rangle)$ for every $x\in F$.''
Using the inclusion-exclusion principle,
we can show
by induction on each element $x$ of a linear extension of $F$
that
\begin{math}
  r(x) = \nabla_A\phi(x)
  = m(\langle x\rangle) - m\left(\langle A \rangle\right)
\end{math}
for the collection $A$ of elements covered by $x$;
thus we obtain (\ref{cm.phi}).
Even if $\phi$ is neither nonnegative nor completely monotone,
the above construction of $m$
(which becomes a signed measure in general)
uniquely satisfies (\ref{cm.phi})
on a finite semilattice $F$,
and $r$ is known as the \emph{M\"{o}bius inverse} of $\phi$.

\begin{theorem}\label{m.cm}
There exists a unique Radon measure $\mu$ on $\mathcal{F}$
that represents $\varphi$
if and only if
$\varphi$ is a completely monotone Scott function on $\mathbb{K}$.
\end{theorem}

Theorem~\ref{m.cm} is the first in a series of results, which we collectively
call ``Choquet theorems on domains.''
The existence of $\mu$ in Theorem~\ref{m.cm} will be proved first,
followed by the proof of Lemma~\ref{mu.unique} for the uniqueness.

\begin{proof}[The existence part of Theorem~\ref{m.cm}]
Let $F$ be a finite subsemilattice of $\mathbb{K}$
(possibly generated by a finite subset of $\mathbb{K}$),
and let $r$ be the M\"{o}bius inverse
of $\varphi$ restricted on $F$.
Since $\mathcal{F}$ is a Scott open base for $\mathbb{K}$,
we can select a collection of distinct elements
$V_z$'s from $\mathcal{F}$
so that $F\cap V_z = F\cap\langle z\rangle^*$ for each $z\in F$.
Then we can define a discrete measure $\mu$ on $\mathcal{F}$
by setting $\mu(\{V_z\}) = r(z)$, $z\in F$,
and obtain
$\mu(\mathcal{F}_x) = \sum_{z\le x} r(z) = \varphi(x)$
for each $x\in F$.
Hence, $\varphi$ can be approximated over $\mathcal{F}$,
and $\mu$ is obtained by Theorem~\ref{s.approx}.
\end{proof}

Assuming the representation of $\mu$ in Theorem~\ref{m.cm}
above we can obtain
\begin{equation}\label{nabla.varphi}
\nabla_{z_1,\ldots,z_n}\varphi(x) =
\mu(\mathcal{F}_x^{z_1,\ldots,z_n}),
\end{equation}
where
\begin{equation}\label{F.x}
\mathcal{F}_x^{z_1,\ldots,z_n}
:= \{V\in\mathcal{F}: x\in V, z_i\not\in V,
i=1,\ldots,n\}.
\end{equation}
Representation (\ref{nabla.varphi}) implies
the necessity of complete monotonicity of $\varphi$
in Theorem~\ref{m.cm}.
The uniqueness of $\mu$ is implied by (\ref{nabla.varphi})
and $\varphi(x) = \mu(\mathcal{F}_x)$,
which is the immediate consequence of
the lemma below.

\begin{lemma}\label{mu.unique}
A Radon measure $\mu$ on $\mathcal{F}$
is uniquely determined by
the measure
on the collection of subsets
$\mathcal{F}_x$ and $\mathcal{F}_x^{z_1,\ldots,z_n}$.
\end{lemma}

\begin{proof}
The open filters $\mathcal{F}_x$, $x\in \mathbb{K}$,
and the closed upper subsets
$$
\langle V_1,\ldots,V_n\rangle_{\mathcal{F}}^*
:= \{V\in\mathcal{F}: V_i\subseteq V \mbox{ for some $i$}\},
\quad
V_1,\ldots,V_n\in\mathcal{F},
$$
generate an open base that consists of subsets of the form
$\mathcal{F}_x$ or
\begin{math}
  \mathcal{F}_x\setminus
  \langle V_1,\ldots,V_n\rangle_{\mathcal{F}}^*
\end{math}
with $n\ge 1$.
Consider a net
\begin{math}
  \left\{
  \langle \llangle z_1\rrangle^*,
  \ldots,\llangle z_n\rrangle^* \rangle_{\mathcal{F}}^*
  \right\}
\end{math}
of closed upper subsets of $\mathcal{F}$
indexed by
$(z_1,\ldots,z_n)\in V_1\times\cdots\times V_n$,
where the indices are inversely ordered coordinate-wise
[i.e., $(z_1,\ldots,z_n)\preceq(z_1',\ldots,z_n')$
if $z_i\ge z_i'$ for each $i$].
Then the net is decreasing and converges to
$\langle V_1,\ldots,V_n\rangle_{\mathcal{F}}^*$.
By MCT we obtain
\begin{equation}\label{mu.subbase}
\mu\left(\mathcal{F}_x\setminus
\langle V_1,\ldots,V_n\rangle_{\mathcal{F}}^*\right)
= \sup
\mu\left(\mathcal{F}_x\setminus
\langle \llangle z_1\rrangle^*,
\ldots,\llangle z_n\rrangle^* \rangle_{\mathcal{F}}^*\right),
\end{equation}
where the supremum is
over $(z_1,\ldots,z_n)\in V_1\times\cdots\times V_n$;
in fact,
(\ref{mu.subbase}) equals the supremum of
$\mu(\mathcal{F}_x^{z_1,\ldots,z_n})$
over the same range.

The measure $\mu$ of the intersection
$$
\left(\mathcal{F}_x\setminus
\langle V_1,\ldots,V_n\rangle_{\mathcal{F}}^*\right)
\cap
\left(\mathcal{F}_y\setminus
\langle U_1,\ldots,U_m\rangle_{\mathcal{F}}^*\right)
=
\mathcal{F}_{x\wedge y}\setminus
\langle V_1,\ldots,V_n,U_1,\ldots,U_m\rangle_{\mathcal{F}}^*
$$
is determined by (\ref{mu.subbase}),
and so is that of the finite union
of the form $\mathcal{F}_x\setminus
\langle V_1,\ldots,V_n\rangle_{\mathcal{F}}^*$
by the inclusion-exclusion principle.
An open subset of $\mathcal{F}$
can be expressed as a union of these open subsets,
and is uniquely determined by MCT.
% do you really need the MCT here?
\end{proof}

The following characterization
is known in the general semi-group setting
(see, e.g.,~\cite{berg}).

\begin{lemma}\label{ca.char}
$\Phi$ is a completely alternating conjugate if and only if
$\varphi_a(x) = \Phi(x\wedge a) - \Phi(x)$
is a completely monotone Scott function for every $a\in \mathbb{K}$.
\end{lemma}

\begin{proof}
Assuming that $\Phi$ is a completely alternating conjugate,
$\varphi_a = -\nabla_a\Phi$ is clearly a completely monotone Scott
function.
Conversely, suppose that
$\varphi_a$ is a completely monotone Scott function.
Then we find $\nabla_B\Phi(x) = -\nabla_B \varphi_a(x)$
if $a \le \bigwedge B\wedge x$,
implying that $\Phi$ is a completely alternating conjugate.
\end{proof}

Let $\Phi$ be a completely alternating conjugate Scott function.
By Theorem~\ref{m.cm}
we can construct
a representation $\lambda_a\in M^+(\mathcal{F})$ for
the completely monotone Scott function $\varphi_a$
of Lemma~\ref{ca.char}.
Since $\varphi_a$ is nondegenerate,
$\lambda_a$ can be viewed as a Radon measure on
the LCH space $\mathcal{F}' = \mathcal{F}\setminus\{\mathbb{K}\}$
(see Definition~\ref{degen.def}).
Observe that
\begin{equation}\label{lambda.a}
  \lambda_a(\mathcal{F}^x)
  = \varphi_a(\hat{1}) - \varphi_a(x)
  = \Phi(a)+\Phi(x)-\Phi(x\wedge a)\le \Phi(x),
\end{equation}
where the equality holds if $a \ge^* x$.
Thus, Theorem~\ref{t.approx}
assures the existence of the measure $\lambda$
for the following version of Choquet theorem.

\begin{theorem}\label{m.ca}
There exists a unique Radon measure $\lambda$ on $\mathcal{F}'$
such that $\Phi(x) = \lambda(\mathcal{F}^x)$ for any $x\in \mathbb{K}$
if and only if $\Phi$ is a completely alternating conjugate.
\end{theorem}

\begin{proof}
If such a measure $\lambda$ exists,
it is easily verified that
\begin{equation}\label{nabla.lambda}
\nabla_{z_1,\ldots,z_n}\Phi(x) =
-\lambda(\mathcal{F}_x^{z_1,\ldots,z_n}),
\end{equation}
which implies that $\Phi$ is complete alternating.
An open base for the LCH $\mathcal{F}'$ consists of open subsets of
the form
$\mathcal{F}_x\setminus\langle V_1,\ldots,V_n\rangle_{\mathcal{F}}^*$
with $n\ge 1$,
and similarly to the proof of Lemma~\ref{mu.unique}
the uniqueness of $\lambda$
is implied by (\ref{nabla.lambda}).
\end{proof}

\begin{remark}\label{m.ca.rem}
When $\Phi$ is normalized
[i.e., $\sup_{x\in \mathbb{K}}\Phi(x) = 1$],
we can introduce
a nondegenerate completely monotone Scott function
$\varphi(x) = 1 - \Phi(x)$, $x\in \mathbb{K}$.
Then the representation $\mu$ of Theorem~\ref{m.cm} for $\varphi$
satisfies $\mu(\{\mathbb{K}\}) = 0$,
and coincides with $\lambda$ of Theorem~\ref{m.ca} for $\Phi$
by uniqueness.
The normalization
implies that $\lambda(\mathcal{F}') = 1$;
in general, we have
$\lambda(\mathcal{F}') = \sup_{x\in \mathbb{K}}\Phi(x)$.
\end{remark}

Recall the class $\mathcal{K}$ of compact sets in
Example~\ref{class.k}.
By means of Theorem~\ref{hofmann.mislove}
we can identify $\mathcal{F}'$
with the class of nonempty closed subsets of $R$.
Thus, the representation $\lambda$ of $\Phi$ in Theorem~\ref{m.ca}
corresponds to that of Theorem~\ref{choquet.1}.

\section{Choquet theorems for sup-difference operators}
\label{sup.mono}

Recall that
$\check{\mathbb{K}} = \mathbb{K}\cup\{\hat{0}\}$
is the one-point compactification of $\mathbb{K}$,
and that $\check{\mathcal{S}} = \mathrm{Scott}(\mathbb{K})$ is a continuous
lattice with the bottom element $\varnothing$.
In the representation theorems below
we consider the subset
$\check{\mathcal{Q}}
:= \{\check{\mathbb{K}}\setminus\langle z\rangle: z\in \check{\mathbb{K}}\}$
of $\check{\mathcal{S}}$,
and introduce an open subbase for $\check{\mathcal{Q}}$ consisting of
$\check{\mathcal{Q}}_x = \{\check{\mathbb{K}}\setminus\langle z\rangle:
x\not\le z\}$,
$x\in \mathbb{K}$,
and $\check{\mathcal{Q}}
\setminus\langle U\rangle^*$,
$U\in\mathcal{S}$.

\begin{lemma}\label{p.homeo}
The map
$\Psi(x) = \check{\mathbb{K}}\setminus\langle x\rangle$
is a homeomorphism
between $\check{\mathbb{K}}$ and $\check{\mathcal{Q}}$.
\end{lemma}

\begin{proof}
Given any $U\in\mathcal{S}$,
we have $z\in U$ if and only if
$U\not\subseteq \check{\mathbb{K}}\setminus\langle z\rangle$,
or equivalently,
$\Psi(U) =
\check{\mathcal{Q}}\setminus\langle U\rangle^*$.
Similarly for any $x\in \mathbb{K}$ we have
$z\in \check{\mathbb{K}}\setminus\langle x\rangle^*$
if and only if $x\in\check{\mathbb{K}}\setminus\langle z\rangle$,
which implies that
$\Psi(\check{\mathbb{K}}\setminus\langle x\rangle^*)
= \check{\mathcal{Q}}_x$.
Together
we have shown that $\Psi$ is a homeomorphism.
\end{proof}

We dually define the operator $\Delta_A$ on a $\vee$-semilattice,
and call it a \emph{successive $\vee$-difference}.
It can be constructed with the $\vee$-difference operator
$\Delta_{z_1}\phi(x) = \phi(x) - \phi(x\vee z_1)$,
and recursively by
$\Delta_{z_1,\ldots,z_n} \phi
= \Delta_{z_n}(\Delta_{z_1,\ldots,z_{n-1}} \phi)$
for $n = 2,3,\ldots$.

\begin{definition}
An increasing function $\phi$ is called \emph{completely $\vee$-alternating}
if $\Delta_A \phi \le 0$ holds for any nonempty finite subset $A$.
Similarly a decreasing function $\phi$ is called \emph{completely $\vee$-monotone}
if $\Delta_A \phi \ge 0$ holds.
\end{definition}

Consider the complete semilattice
$\check{\mathbb{K}}' := \check{\mathbb{K}}\setminus\{\hat{1}\}$.
Then the Lawson topology of $\check{\mathbb{K}}'$ is compact
since $\{\hat{1}\}$ is an open subset of $\check{\mathbb{K}}$.
In fact, any continuous complete semilattice
(i.e., complete semilattice which is also a domain)
is compact Hausdorff in the Lawson topology
(cf. Section~III-1 of~\cite{scott}).
The map $\Psi$ in Lemma~\ref{p.homeo} is homeomorphic from
$\check{\mathbb{K}}'$ to the compact subset
$\mathcal{Q} := \check{\mathcal{Q}}\setminus\{\varnothing\}$
of $\mathcal{S}$.

\begin{lemma}\label{m.sup-ca}
A completely $\vee$-alternating Scott function
$\varphi$ on $\mathbb{K}$ has a representation $\mu$ on $\mathcal{Q}$.
\end{lemma}

\begin{proof}
We can naturally extend $\varphi$
to the Scott function $\check{\varphi}$ on $\check{\mathbb{K}}$
by setting $\check{\varphi}(\hat{0}) = 0$.
For any nonempty finite subset $A$ of $\mathbb{K}$
we can observe that
\begin{math}
  \Delta_A\check{\varphi}(\hat{0})
  \le \Delta_A\varphi\left(\bigwedge A\right) \le 0
\end{math}
where $\bigwedge A$ is the greatest lower bound of $A$;
thus $\check{\varphi}$ is also completely $\vee$-alternating.
Let $\check{F}$ be a finite sup-subsemilattice of $\check{\mathbb{K}}$,
and let $\check{\varphi}^*(x) = 1 - \check{\varphi}(x)$
be a completely monotone function on the dual poset $\check{F}^*$.
Hence we can introduce the M\"{o}bius inverse $r^*$
of $\check{\varphi}^*$ on $\check{F}^*$.
Without loss of generality we assume
$\hat{0}, \hat{1} \in \check{F}$,
and note that $r^*(\hat{1}) = 0$.

For $F = \check{F}\setminus\{\hat{0}\}$,
we can construct a discrete measure $\mu_F$ on $\mathcal{Q}$
by setting $\mu_F(\{\check{\mathbb{K}}\setminus\langle z\rangle\}) = r^*(z)$
for each $z\in \check{F}\setminus\{\hat{1}\}$.
Then we obtain
\begin{equation*}
  \mu_F(\mathcal{Q}_x)
  = \sum_{z\in \check{F}\setminus\langle x\rangle^*} r^*(z)
  = \check{\varphi}^*(\hat{0}) - \check{\varphi}^*(x)
  = \varphi(x)
\end{equation*}
for each $x\in F$.
Hence, $\varphi$ is approximated over $\mathcal{Q}$,
and it has a representation $\mu$ on $\mathcal{Q}$
by Theorem~\ref{s.approx}.
\end{proof}

Let $\Phi$ be a completely $\vee$-monotone conjugate on $\mathbb{K}$.
Then $\varphi_a(x) = \Phi(a) - \Phi(a\vee x)$
is a completely $\vee$-alternating Scott function
for each $a\in \mathbb{K}^*$,
and the corresponding representation $\lambda_a$ of $\varphi_a$ over
$\mathcal{Q}$ may be viewed as a measure on
the LCH $\mathcal{Q}' := \mathcal{Q}\setminus\{\mathbb{K}\}$
since $\varphi_a$ is nondegenerate (see Definition~\ref{degen.def}).
Similarly to (\ref{lambda.a}) we can observe that
$\lambda_a(\mathcal{Q}^x) = \Phi(a\vee x)\le\Phi(x)$,
for which the equality holds if $a \ge^* x$.
Hence, by Theorem~\ref{t.approx} we have established

\begin{corollary}\label{m.sup-cm}
For any completely $\vee$-monotone conjugate $\Phi$ on $\mathbb{K}$,
there exists some $\lambda\in M^+(\mathcal{Q}')$ such that
$\Phi(x) = \lambda(\mathcal{Q}^x)$ for any $x\in \mathbb{K}$.
\end{corollary}

Let $\Psi$ be the homeomorphism of Lemma~\ref{p.homeo}.
Then the measure $\mu$ of Lemma~\ref{m.sup-ca}
induces a Radon measure $\Lambda$ on
the compact space $\check{\mathbb{K}}'$
by setting $\Lambda(B) = \mu(\Psi(B))$
for any Borel-measurable subset $B$ of $\check{\mathbb{K}}'$.
Clearly it satisfies
$\Lambda(\check{\mathbb{K}}\setminus\langle x\rangle^*) = \varphi(x)$,
$x\in \mathbb{K}$.
Conversely,
$\varphi(x) = \Lambda(\check{\mathbb{K}}\setminus\langle x \rangle^*)$
is an increasing function on $\mathbb{K}$,
and
$\Delta_A\varphi(x) =
-\Lambda(\langle x\rangle^*\setminus\langle A\rangle^*)$
implies that
$\varphi$ is completely $\vee$-alternating.
An increasing net
$\check{\mathbb{K}}\setminus\langle z \rangle^*$,
$z\in\llangle x \rrangle$,
converges to
$\check{\mathbb{K}}\setminus\langle x \rangle^*$,
and therefore,
by MCT we obtain
$\Lambda(\check{\mathbb{K}}\setminus\langle x \rangle^*)
= \sup_{z\in\llangle x \rrangle}
\Lambda(\check{\mathbb{K}}\setminus\langle z \rangle^*)$.
Thus, $\varphi$ is a Scott function.

Similarly we can associate the measure $\lambda$ of
Corollary~\ref{m.sup-cm} with the Radon measure $\Lambda(B)
= \lambda(\Psi(B))$ on
the LCH $\mathbb{K}' := \mathbb{K}\setminus\{\hat{1}\}$, and obtain
$\Lambda(\langle x\rangle^*) = \Phi(x)$
and $\Lambda(\langle x\rangle^*\setminus\langle A\rangle^*)
= \Delta_A\Phi(x)$.
Conversely,
$\Phi(x) = \Lambda(\langle x\rangle^*)$
is a completely $\vee$-monotone conjugate.

\begin{lemma}\label{R-rep.unique}
(i) A Radon measure $\Lambda$ on the compact space $\check{\mathbb{K}}'$
is uniquely determined by $\Lambda$ on the collection of subsets
$\check{\mathbb{K}}\setminus\langle x\rangle^*$ and
$\langle x\rangle^*\setminus\langle A\rangle^*$
where $x\in \mathbb{K}$ and $A$ is a nonempty finite subset of $\mathbb{K}$.

(ii) A Radon measure $\Lambda$ on the LCH space $\mathbb{K}'$
is uniquely determined by the values of $\Lambda$ on the collection of subsets
$\langle x\rangle^*\setminus\langle A\rangle^*$
where $x\in \mathbb{K}$ and $A$ is a nonempty finite subset of $\mathbb{K}$.
\end{lemma}

\begin{proof}
Observe that
$U\setminus\langle A\rangle^*$
with open filter $U$ of $\check{\mathbb{K}}$ (respectively, of $\mathbb{K}$)
and nonempty finite subset $A$ of $\mathbb{K}$
forms an open base for the Lawson topology
of $\check{\mathbb{K}}'$ (respectively, of $\mathbb{K}'$).
In what follows we only present the proof of (i),
the proof of (ii) being similar.

If $U = \check{\mathbb{K}}$ then after setting $x=\bigwedge A$
the set $\check{\mathbb{K}}\setminus\langle A\rangle^*$
is a disjoint union of
$\check{\mathbb{K}}\setminus\langle x\rangle^*$ and
$\langle x\rangle^*\setminus\langle A\rangle^*$.
Otherwise, we can assume $U\neq\check{\mathbb{K}}$,
and consider a net
\begin{math}
  \left\{
  \llangle z\rrangle^*\setminus\langle A\rangle^*
  \right\}
\end{math}
of open subsets of $\mathbb{K}$ indexed by $z\in U$,
where the index set $U$ is equipped with the dual order $\le^*$.
Then the net is increasing, and converges to
$U\setminus\langle A\rangle^*$.
By MCT we can show that
\begin{math}
  \Lambda(U\setminus\langle A\rangle^*)
  = \sup_{z\in U} \Lambda(\llangle z\rrangle^*\setminus
  \langle A\rangle^*),
\end{math}
which is equal to
\begin{math}
  \sup_{z\in U} \Lambda(\langle z\rangle^*\setminus
  \langle A\rangle^*).
\end{math}
Thus, the measure $\Lambda$ on the open subset $U\setminus\langle A\rangle^*$
is uniquely determined as in~(i).
By the inclusion-exclusion principle,
the measure $\Lambda$ of the finite union
$\bigcup_{i=1}^n U_i\setminus\langle A_i\rangle^*$
can be expressed in terms of the intersection
$\bigcap_{i=1}^n U_i\setminus\langle A_i\rangle^*=\left(\bigcap_{i=1}^n U_i\right)
\setminus\left\langle \bigcup_{i=1}^n A_i\right\rangle^*$
with open filter $\bigcap_{i=1}^n U_i$.
Hence, a Radon measure $\Lambda$
is uniquely extended to the collection
of open subsets of $\check{\mathbb{K}}'$ by MCT.
\end{proof}

Lemma~\ref{R-rep.unique} together with Lemma~\ref{m.sup-ca} and
Corollary~\ref{m.sup-cm} completes the proof of Theorem~\ref{R-rep}.
The version of Choquet theorem for completely $\cap$-monotone
capacities in Section~\ref{choquet.intro}
is a special case of Theorem~\ref{R-rep}(ii)
where the lattice $\mathbb{K}$ of Example~\ref{class.k} is
considered.

\begin{theorem}\label{R-rep}
(i)
There exists a unique Radon measure $\Lambda$ on
$\check{\mathbb{K}}'$
such that
$\varphi(x)
= \Lambda(\check{\mathbb{K}}\setminus\langle x\rangle^*)$,
$x\in \mathbb{K}$,
if and only if
$\varphi$ is a completely $\vee$-alternating Scott function.

(ii)
There exists a unique Radon measure $\Lambda$ on $\mathbb{K}'$
such that $\Phi(x) = \Lambda(\langle x\rangle^*)$,
$x\in \mathbb{K}$,
if and only if
$\Phi$ is a completely $\vee$-monotone conjugate.
\end{theorem}

\begin{remark}
A decreasing net
$\check{\mathbb{K}}\setminus\llangle z\rrangle^*$,
$z\in \mathbb{K}^*$,
of closed sets converges to $\{\hat{0}\}$.
Thus, a Radon measure $\Lambda$ satisfies
$\Lambda(\{\hat{0}\}) = \inf_{z\in \mathbb{K}^*}
\Lambda(\check{\mathbb{K}}\setminus\llangle z\rrangle^*)$,
which equals
$\inf_{z\in \mathbb{K}^*}
\Lambda(\check{\mathbb{K}}\setminus\langle z\rangle^*)
= \inf_{z\in \mathbb{K}} \varphi(z)$
in the context of Theorem~\ref{R-rep}(i).
Consequently, $\Lambda$ is nonatomic at $\hat{0}$
if $\varphi$ is nondegenerate.
Similarly in Theorem~\ref{R-rep}(ii)
we can show that $\Lambda(\mathbb{K}') = \sup_{z\in \mathbb{K}} \Phi(z)$.
\end{remark}

\section{Valuations on a distributive lattice}
\label{val.sec}

A lattice $\mathbb{L}$ is distributive if
for any nonempty finite subset $A$ of $\mathbb{L}$ we have
\begin{equation}\label{distributive}
x \wedge \bigvee A = \bigvee (x\wedge A),
\end{equation}
where $\bigvee A$ denotes the least upper bound of $A$
and $x\vee A := \{x\vee z: z\in A\}$.
The distributivity of \eqref{distributive} is dually characterized by
$x \vee \bigwedge A = \bigwedge (x\vee A)$.
Given any finite subset $G$ of $\mathbb{L}$
the distributivity
allows us to construct the finite distributive sublattice $F$
by first generating the sup-subsemilattice $H$
by all elements of the form $\bigvee A$,
$\varnothing\neq A\subseteq G$,
then extending $H$ to the sublattice $F$
which consists of
all elements of the form $\bigwedge B$,
$\varnothing\neq B\subseteq H$.

In a distributive lattice $\mathbb{L}$ the following statements are
equivalent for $z \neq\hat{1}$
(cf. Section~I-3 of~\cite{scott}).
\begin{itemize}
\item[(a)]
$\mathbb{L}\setminus\langle z\rangle$ is a filter;
\item[(b)]
$z = x\wedge y$ implies $z = x$ or $z = y$;
\item[(c)]
$z$ is maximal in $\mathbb{L}\setminus U$ with some open filter $U$.
\end{itemize}
An element $z$ is called \emph{prime}
[or, \emph{$\wedge$-irreducible}]
if it satisfies either (a) or (c)
[respectively, if it satisfies (b)].
The top element $\hat{1}$ satisfies neither (a) nor (c)
while (b) holds for $z = \hat{1}$;
thus, there is a subtle distinction between
prime and $\wedge$-irreducible elements.
The following result is a straightforward consequence of
(c) along with the $T_0$-property of Scott topology
of a domain
(cf. Theorem I-3.7 of \cite{scott} for the proof).

\begin{proposition}\label{abundant}
Assume that $\mathbb{L}$ is a domain.
Then if $x \not\le y$
then there exists some prime element $z$ of $\mathbb{L}$
such that $x \not\le z$ and $y \le z$.
\end{proposition}

In the rest of our investigation
we extend Assumption~\ref{k.ass}
for the domain $\mathbb{K}$ to be distributive.
The lattice $\mathcal{K}$ of Example~\ref{class.k}
is distributive,
and moreover, the following characterization of
prime elements of $\mathcal{K}$ follows
(cf. Example~I-3.14 of~\cite{scott}).

\begin{proposition}\label{prime.q}
A compact set $Q$ is prime in $\mathcal{K}$ if and only if $Q$ is
a singleton.
\end{proposition}

\begin{proof}
A singleton is obviously prime.
Conversely, suppose $Q$ is prime in $\mathcal{K}$.
Because $\mathcal{U} = \{E\in\mathcal{K}: Q\setminus E\neq\varnothing\}$
is a filter in $\mathcal{K}$,
the corresponding filter base
$\mathcal{B} = \{Q\setminus E: E\in\mathcal{U}\}$
on the compact Hausdorff space $Q$
has a cluster point, say $a \in Q$.
If $E$ is a compact neighborhood of $a$ in the LCH space
$R$ then $Q\subseteq E$; otherwise, $Q\setminus E\in\mathcal{B}$,
and the set of cluster points must be contained in
$Q\setminus \mathrm{int}(E)$, which contradicts
$a\not\in Q\setminus \mathrm{int}(E)$.
Thus, we conclude that $Q = \{a\}$.
\end{proof}

By $P$
we denote the collection of prime elements in $\mathbb{K}$.
The continuous lattice
$\check{\mathbb{K}}$ is also multiplicative and distributive,
and $\check{P} := P\cup\{\hat{0}\}$
is the corresponding collection of prime elements in $\check{\mathbb{K}}$.
We can observe that
$$
\mathcal{P} := \mathcal{Q}\cap\mathcal{F}
= \left\{\check{\mathbb{K}}\setminus\langle z\rangle:
z\in \check{P}\right\}
$$
is a compact subset of $\mathcal{S}$,
and isomorphic to $\check{P}$.
For any positive integer $k$
we can introduce a continuous map $\Pi_k$
from $(V_1,\ldots,V_k)\in\mathcal{P}^k$ to
\begin{math}
  \Pi_k(V_1,\ldots,V_k) := \bigcap_{i=1}^k V_i
  \in\mathcal{F}.
\end{math}
Since the product space $\mathcal{P}^k$ is compact,
so is the image
\begin{equation*}
  \Pi_k\left(\mathcal{P}^k\right)
  = \left\{\check{\mathbb{K}}\setminus\langle A\rangle:
  A\subseteq \check{P}, 1 \le |A| \le k
  \right\},
\end{equation*}
which we denote by $\mathcal{P}_k$.
It should be noted that
\begin{math}
  \mathbb{K} = \check{\mathbb{K}}\setminus\langle\hat{0}\rangle
  \in\mathcal{P}_k .
\end{math}

Recall the notation of~(\ref{F.x});
for any $w\in \mathbb{K}$ and any finite subset $B$ of $\mathbb{K}$,
we will write
\begin{equation}\label{f.b}
\mathcal{F}_w^B :=
\{V\in\mathcal{F}: w\in V,\, V\cap B = \varnothing\},
\end{equation}
in which $\mathcal{F}_x^B = \mathcal{F}_x$ if $B = \varnothing$.
Then we can observe the following property.

\begin{lemma}\label{p.k.lemma}
If a finite subset $B$ of $\mathbb{K}$ satisfies $|B| \ge k+1$ then
$\mathcal{P}_k\cap\mathcal{F}_w^{B}=\varnothing$
whenever
$w \le {\bigwedge_{\{x,y\}\subseteq B} x\vee y}$.
\end{lemma}

\begin{proof}
Suppose that
there is some $A\subseteq \check{P}$
such that $1 \le |A| \le k$ and
$\check{\mathbb{K}}\setminus\langle A\rangle\in\mathcal{F}_w^{B}$.
Since $B\subseteq\langle A\rangle$ and $|A| < |B|$,
there exists a pair $\{x,y\}\subseteq B$
such that $x,y\le z$ for some $z\in A$.
But it implies that $w \le x\vee y \le z$, and therefore,
that
$\check{\mathbb{K}}\setminus\langle A\rangle\not\in\mathcal{F}_w^{B}$,
which is a contradiction.
\end{proof}

\begin{remark}\label{p.k.rem}
Recall that in the setting of Theorem~\ref{hofmann.mislove}
the lattice $\mathcal{F}$ of closed sets
is homeomorphic to $\mathrm{OFilt}(\mathcal{K})$,
and that
the element $\varnothing$ of $\mathcal{F}$
corresponds to the top element $\mathcal{K}$ of
$\mathrm{OFilt}(\mathcal{K})$.
In the context of Section~\ref{val.intro}
the collection
$\{A\in\mathcal{F}: 0\le |A|\le k\}$
of finite subsets of at most size~$k$ in $R$
is isomorphic to the compact subset
$\mathcal{P}_k$ of $\mathrm{OFilt}(\mathcal{K})$,
and it will be referred to by the same symbol $\mathcal{P}_k$.
In particular, $\mathcal{P}$
($=\mathcal{P}_1$) consists of all the singletons and the empty set
$\varnothing$, and it is isomorphic to the one-point
compactification of $R$.
For any finite subset $B = \{Q_1,\ldots,Q_n\}$ of $\mathcal{K}$
and $W\in\mathcal{K}$ satisfying
$W \supseteq \bigcup_{1\le i<j\le n}Q_i\cap Q_j$,
we can express (\ref{f.b}) by
\begin{equation*}
  \mathcal{F}_W^B
  = \{F\in\mathcal{F}: F\cap W = \varnothing,
  F\cap Q_i \neq\varnothing
  \mbox{ for all $i=1,\ldots,n$}\} .
\end{equation*}
Thus, if $F\in\mathcal{F}_W^B$
then $F$ must contain
at least one point for each disjoint
sequence $Q_1\setminus W,\ldots,Q_n\setminus W$,
therefore
$F\not\in\mathcal{P}_k$ for $k < n$.
This validates Lemma~\ref{p.k.lemma}.
\end{remark}

\begin{definition}\label{val.def}
For any positive integer $k$,
a completely monotone Scott function $\varphi$ on $\mathbb{K}$
is called a \emph{$k$-valuation} if
\begin{equation}\label{k.val}
\nabla_B\varphi\left(
\bigwedge_{\{x,y\}\subseteq B} x\vee y
\right) = 0
\end{equation}
holds for any ${(k+1)}$-element subset $B$ of $\mathbb{K}$.
Similarly, a completely alternating conjugate $\Phi$ is called
a \emph{conjugate $k$-valuation}
if (\ref{k.val}) holds
for any ${(k+1)}$-element subset $B$ of $\mathbb{K}$.
\end{definition}

\begin{remark}\label{incomp.rem}
We note in (\ref{k.val}) that the greatest lower bound
\begin{equation*}
  o_B = \bigwedge_{\{x,y\}\subseteq B} x\vee y
\end{equation*}
is considered for all the joins $(x\vee y)$'s of
distinct pair $\{x,y\}$ from $B$.
If $B$ contains some comparable pair,
say $x_1<y_1$, then we have $o_B\le y_1$
and by setting $B_1=B\setminus\{y_1\}$
we can find
\begin{math}
  \nabla_B\varphi(o_B)
  = \nabla_{B_1}\varphi(o_B) - \nabla_{B_1}\varphi(o_B\wedge y_1)
  = 0 .
\end{math}
Thus, it suffices to check
\eqref{k.val} only for antichains $B$'s
[i.e., $B$'s consisting of pairwise incomparable elements].
\end{remark}

\begin{proposition}\label{valuation}
Let $\phi$ be an increasing or a decreasing function on $\mathbb{K}$.
The 1-valuation condition
\begin{equation}\label{module}
\phi(x) + \phi(y) = \phi(x\wedge y) + \phi(x\vee y),
\quad x,y\in \mathbb{K},
\end{equation}
is equivalent to:
\begin{itemize}
\item[(i)]
$\phi$ is completely monotone and completely $\vee$-alternating
if it is increasing; or
\item[(ii)]
$\phi$ is completely alternating and completely $\vee$-monotone
if it is decreasing.
\end{itemize}
\end{proposition}

\begin{proof}
We can observe that
\begin{equation*}
\phi(x\wedge y) + \phi(x\vee y)
- \phi(x) - \phi(y)
= \nabla_{x,y}\phi(x\vee y)
= \Delta_{x,y}\phi(x\wedge y) .
\end{equation*}
Then (i) implies that
$\nabla_{x,y}\phi(x\vee y) \ge 0$
and $\Delta_{x,y}\phi(x\wedge y) \le 0$,
and therefore, that $\phi$ satisfies (\ref{module}).
Similarly (ii) implies (\ref{module}).
Conversely, suppose that (\ref{module}) holds.
Then for any nonempty finite subset $A$ of $\mathbb{K}$
we can deduce
$\nabla_A\phi(x)
= \phi(x) - \phi(x \wedge \bigvee A)$
and
$\Delta_A\phi(x)
= \phi(x) - \phi(x \vee \bigwedge A)$
by induction,
which implies either (i) or (ii).
\end{proof}

When $\varphi$ is a $1$-valuation in Definition~\ref{val.def},
we simply call it \emph{valuation} (or \emph{module}).
By Proposition~\ref{valuation}
Scott function $\varphi$ or conjugate Scott function $\Phi$ is a
valuation or a conjugate valuation respectively
if (\ref{module}) holds for $\varphi$ or $\Phi$;
there is no need to check complete monotonicity or completely
alternating property for $1$-valuation.

\begin{proposition}\label{extend.val}
A $k$-valuation $\varphi$
is also $k'$-valuation for every $k' > k$.
\end{proposition}

\begin{proof}
Let $B'$ be an arbitrary $(k'+1)$-element antichain,
and let $F$ be a finite sublattice generated by $B'$.
Since $\varphi$ is completely monotone,
we can construct a measure
$m(A) = \sum_{x\in A} r(x)$, $A\subseteq F$,
with the M\"{o}bius inverse $r$ of
the restriction of $\varphi$ to $F$.
Choose any $(k+1)$-element subset $B$ of $B'$,
and observe that
$z = \bigwedge_{\{x,y\}\subseteq B} x\vee y
\ge z' = \bigwedge_{\{x,y\}\subseteq B'} x\vee y$.
By applying (\ref{nabla.r}), we obtain
\begin{equation*}
0 \le \nabla_{B'}\varphi(z')
= m(\langle z'\rangle\setminus\langle B'\rangle)
\le m(\langle z\rangle\setminus\langle B\rangle)
= \nabla_{B}\varphi(z) = 0,
\end{equation*}
where $\langle z\rangle$ and $\langle B\rangle$
are viewed as the lower subsets of $F$ generated by $z$ and $B$.
Thus, $\varphi$ is a $k'$-valuation.
\end{proof}

% * <motoyamachida@gmail.com> 00:00:00 12 Nov 2017 UTC-0500:
% Comment out: Similarly we can show that dual $k$-valuations are
% dual $k'$-valuations for each $k' > k$. It should be noted,
% however, that $k$-valuation $\varphi$ is not necessarily
% $k'$-valuation when $k' < k$.

Suppose that $\varphi$ is represented by $\mu$ over $\mathcal{P}_k$.
Recalling (\ref{nabla.varphi}) from Section~\ref{sec.cm},
we find (\ref{k.val}) by Lemma~\ref{p.k.lemma},
which leads to the following version of Choquet theorem.

\begin{theorem}\label{val.rep}
There exists a unique representation
$\mu$ of $\varphi$ over $\mathcal{P}_k$
if and only if
$\varphi$ is a $k$-valuation.
\end{theorem}

Lemma~\ref{p.k.lemma} implies the necessity of $k$-valuation
in Theorem~\ref{val.rep}.
In the next two auxiliary lemmas
we prepare the proof of sufficiency.

For any finite sublattice $F$ of $\mathbb{K}$,
we can introduce the collections
$J_F$ and $J'_F$ respectively
of $\wedge$-irreducible elements
and of prime elements in $F$.
We denote the bottom and the top element of $F$
by $\hat{0}_F := \bigwedge F$ and $\hat{1}_F := \bigvee F$,
respectively.
Here we have $J'_F = J_F\setminus\{\hat{1}_F\}$,
and view $J_F$ as a subposet of $F$.
According to the fundamental theorem for finite distributive
lattices (e.g., Theorem~3.4.1 of~\cite{stanley}),
the distributive lattice $F$
is poset-isomorphic to the lattice of nonempty upper
subsets of $J_F$,
mapping from $x\in F$ to $J_F\cap\langle x\rangle^*$.

\begin{lemma}\label{partition}
Let $F$ be a finite sublattice of $\mathbb{K}$,
and let $B_F^x$ be the collection of maximal elements of
$F\setminus\langle x\rangle^*$ where $x \neq \hat{0}_F$;
set $B_F^{\hat{0}_F} = \varnothing$.
Then the collection
$\{\mathcal{F}_x^{B_F^x}\}_{x\in F}$
partitions $\mathcal{F}$.
\end{lemma}

We note that $B_F^x\subseteq J'_F$ is an antichain,
and for every $x\in F$ the correspondence $x\mapsto B_F^x\subseteq J'_F$ is one-to-one.
Conversely, any antichain $B$ of $J'_F$ (which is possibly the empty set)
corresponds uniquely to
$x_B := \bigwedge (F\setminus\langle B\rangle)$
in such a way that $B = B_F^{x_B}$,
where we set $\langle B\rangle = \varnothing$
if $B = \varnothing$ for convenience.

\begin{proof}[Proof of Lemma~\ref{partition}]
For any $V\in\mathcal{F}$
and any $x\in F$,
we have $V\in\mathcal{F}_x^{B_F^x}$
if and only if $V\cap\langle B_F^x\rangle = \varnothing$
and $J_F\cap\langle x\rangle^* \subset V$.
Since $\langle B_F^x\rangle\cap J_F =
J_F\setminus\langle x\rangle^*$,
we can find
$V\in\mathcal{F}_x^{B_F^x}$ if and only if
$x = \bigwedge (J_F\cap V)$.
\end{proof}

In the setting of Lemma~\ref{partition}
we can observe that
\begin{equation}\label{b.x}
x \le \bigwedge_{\{y,z\}\subseteq B_F^x} y\vee z,
\end{equation}
and by Lemmas~\ref{p.k.lemma} that
$\mathcal{P}_k\cap\mathcal{F}_x^{B_F^x} = \varnothing$
if $|B_F^x| \ge k+1$.
In what follows
we set $\nabla_B\varphi(x) = \varphi(x)$
if $B = \varnothing$ for convenience,
and obtain the following corollary to Lemma~\ref{partition}.

\begin{corollary}\label{p.k.bound}
Let $\varphi$ be a completely monotone Scott function on $\mathbb{K}$,
and let $F$ be a finite sublattice of $\mathbb{K}$.
Then the representation $\mu$ for $\varphi$
satisfies
$$
\mu\left(\mathcal{P}_k\right)
\le \sum \nabla_{B_F^x}\varphi(x),
$$
where the summation is over all $x\in F$
satisfying $|B_F^x| \le k$.
\end{corollary}

In the following lemma we present a construction
of approximation of $\varphi$ over $\mathcal{P}_k$.
The proof of Lemma~\ref{v.approx}
is preceded by the construction of antichains of $P$,
and followed by the proof of Theorem~\ref{val.rep}.

\begin{lemma}\label{v.approx}
For $\varphi$ and $F$ of Corollary~\ref{p.k.bound},
we can construct $\mu_F\in M^+(\mathcal{F})$
so that it satisfies $\varphi(x) = \mu_F(\mathcal{F}_x)$,
$x\in F$, and
\begin{equation}\label{v.bound}
\mu_F\left(\mathcal{F}\setminus\mathcal{P}_k\right)
\le \sum \nabla_B \varphi\left(
\bigwedge_{\{x,y\}\subseteq B} x\vee y
\right),
\end{equation}
where the summation is over antichains $B$ of
$J'_F$ with $|B| = k+1$.
\end{lemma}

By $\langle x\rangle_F$ and $\langle x\rangle_F^*$
we denote the principal lower and upper set in the lattice $F$, respectively.
For each $q\in J'_F$ the \emph{coprime}
$\bar{q} := \bigwedge(F\setminus\langle q\rangle_F)$
satisfies
$\langle \bar{q}\rangle_F^* = F\setminus\langle q\rangle_F$.
By Proposition~\ref{abundant}
we can choose $z(q)\in P$ satisfying
$\bar{q}\not\le z(q)$ and $q \le z(q)$.
For each element $x\in F\setminus\{\hat{0}_F\}$
the corresponding antichain $B_F^x\subseteq J'_F$
satisfies $x = \bigwedge(F\setminus\langle B_F^x\rangle)$.
Suppose $\{q_1,q_2\}\subseteq B_F^x$.
Then we have $\bar{q}_i \le x$ and $x\not\le z(q_i)$ for $i=1,2$.
Since $x\le q_1\vee q_2 \le z(q_1)\vee z(q_2)$,
$z(q_1)$ and $z(q_2)$ are not comparable.
Thus, $A_x = \{z(q): q\in B_F^x\}\subseteq P$ is an antichain
if $x\neq\hat{0}_F$; set $A_{\hat{0}_F} = \{\hat{0}\}$.

\begin{proof}[Proof of Lemma~\ref{v.approx}]
Let $r$ be the M\"{o}bius inverse of $\varphi$ on $F$,
and let $m(C) = \sum_{x\in C} r(x)$ be the corresponding measure on
$F$.
Similarly to the proof of Theorem~\ref{m.cm},
we can construct a discrete measure $\mu_F$ of $\varphi$
on $\mathcal{F}$ by setting
$\mu_F(\{\check{\mathbb{K}}\setminus\langle A_x\rangle\}) = r(x)$,
$x\in F$.
Then we have
\begin{equation*}
  \mu_F\left(
  \mathcal{F}\setminus\mathcal{P}_k
  \right)
  \le \sum \mu_F(\{\check{\mathbb{K}}\setminus\langle A_x\rangle\})
  = m(C),
\end{equation*}
where the summation is over the subset
$C = \{x\in F: |B_F^x| \ge k+1\}$.

Observe that $C$
is covered by the collection of the subsets
$E_B = \{x\in F: B\subseteq B_F^x\}$
indexed by the antichains $B\subseteq J'_F$ such that $|B| = k+1$.
Let $B$ be such an antichain,
and let $b = \bigwedge_{\{x,y\}\subseteq B} x\vee y$.
For $x\in E_B$ we can observe that 
$x\not\in\langle B_F^x\rangle$,
and (\ref{b.x}) implies $x\le b$;
thus we have
\begin{math}
  x\in\langle b\rangle_F\setminus\langle B_F^x\rangle
  \subseteq
  \langle b\rangle_F\setminus\langle B\rangle .
\end{math}
Since
\begin{math}
  m(E_B) \le m(\langle b\rangle_F\setminus\langle B\rangle)
  = \nabla_{B} \varphi(b) ,
\end{math}
the upper bound of \eqref{v.bound} holds.
\end{proof}

\begin{proof}[Proof of Theorem~\ref{val.rep}]
Let $\varphi$ be a $k$-valuation,
and let $\{\mu_F\}_{F\in\mathbb{F}}$ be 
the net of measures $\mu_F$ of Lemma~\ref{v.approx}
which approximates $\varphi$ over $\mathcal{P}_k$.
Then the resulting representation
$\mu$ of Theorem~\ref{s.approx}
is that of Theorem~\ref{m.cm} by uniqueness.
\end{proof}

The proof of the Choquet theorem for conjugate $k$-valuations
parallels that of Theorem~\ref{m.ca}:
Once we obtain Lemma~\ref{k.char}, we can observe that the Radon
measure $\lambda$ of Theorem~\ref{m.ca} represents $\Phi$ over
the LCH $\mathcal{P}_k' := \mathcal{P}_k\setminus\{\mathbb{K}\}$;
thus, establishing Theorem~\ref{k.ca}.

\begin{lemma}\label{k.char}
$\Phi$ is a conjugate $k$-valuation if and only if
$\varphi_a$ of Lemma~\ref{ca.char}
is a $k$-valuation for every $a\in \mathbb{K}$.
\end{lemma}

\begin{proof}
By the expansion formula of (\ref{nabla.phi})
we obtain
\begin{align*}
& \nabla_{z_1,\ldots,z_k}\varphi_a\left(\bigwedge_{i\neq j} z_i\vee z_j\right)
\\ & =
\nabla_{z_1\wedge a,\ldots,z_k\wedge a}
\Phi\left(\bigwedge_{i\neq j} (z_i\wedge a)\vee (z_j\wedge a)\right)
-
\nabla_{z_1,\ldots,z_k}\Phi\left(\bigwedge_{i\neq j} z_i\vee z_j\right) .
\end{align*}
which implies the $k$-valuation of $\varphi_a$
if $\Phi$ is a conjugate $k$-valuation.
As in Remark~\ref{incomp.rem},
the above successive difference with 
$z_1\wedge a,\ldots,z_k\wedge a$ vanishes
if it contains a comparable pair;
for example, if $a \le z_1\wedge z_2$.
Thus the $k$-valuation of $\varphi_a$
implies that of $\Phi$.
\end{proof}

\begin{theorem}\label{k.ca}
There exists a unique Radon measure $\lambda$
on $\mathcal{P}_k'$
such that $\Phi(x) = \lambda(\mathcal{P}_k^x)$ for any $x\in \mathbb{K}$
if and only if $\Phi$ is a conjugate $k$-valuation.
\end{theorem}

In the context of Section~\ref{val.intro}
(see Remark~\ref{p.k.rem})
the LCH
$\mathcal{P}_k' = \mathcal{P}_k\setminus\{\varnothing\}$
is the collection of non-empty finite
subsets of at most size~$k$ in $R$, so
Theorem~\ref{val.2}
is an immediate corollary to Theorem~\ref{k.ca}.

\section{Locally finite valuations}
\label{lfv.sec}

We say that an open filter $V$ in a domain $\mathbb{L}$
is \emph{$\sigma$-compact}
if there is a countable set $\{w_i\}$ of $V$
such that $V = \bigcup_{i=1}^{\infty}\langle w_i\rangle^*$.
In what follows we assume that
$\mathbb{K}$ is $\sigma$-compact,
and fix a sequence $\{w_i\}\subseteq\mathbb{K}$
that satisfies $\mathbb{K} = \bigcup_{i=1}^{\infty}\langle w_i\rangle^*$
and $w_{i+1}\ll w_i$ for $i=1,2,\ldots$.
Then we can introduce
a continuous map $\Xi_i$
from $F\in\mathcal{F}$ to
$\Xi_i(F) := F\cap\llangle w_i \rrangle^*$,
and the compact subset
$$
\Xi_{i}\left(\mathcal{P}_{k}\right)
= \{\llangle w_i \rrangle^*\setminus\langle A \rangle:
A\subseteq P,\, 0 \le |A| \le k\}
$$
of $\mathcal{F}$.
Thus, we can define the $F_{\sigma}$-set
$\bigcup_{k=1}^{\infty}
\Xi_{i}^{-1}\left(\Xi_{i}\left(\mathcal{P}_{k}\right)\right)$,
and the $F_{\sigma\delta}$-set
\begin{equation}\label{p.lf}
\mathcal{P}_{\mathrm{lf}}
:= \bigcap_{i=1}^{\infty}\bigcup_{k=1}^{\infty}
\Xi_{i}^{-1}\left(\Xi_{i}\left(\mathcal{P}_{k}\right)\right)
\end{equation}

In Example~\ref{class.k}
we can further assume that the LCH space $R$ is $\sigma$-compact.
Then we can construct a sequence $W_i$ of compact subsets of $R$ so that
$R = \bigcup_{i=1}^\infty \mathrm{int}(W_i)$
and $W_i\subseteq\mathrm{int}(W_{i+1})$ for $i=1,2,\ldots$,
therefore the domain $\mathcal{K}$ is $\sigma$-compact.
In light of Theorem~\ref{hofmann.mislove}
we can view $\Xi_{i}$ as a map from $F\in\mathcal{F}$
to the closed set $F\cap\mathrm{int}(W_i)$
on the space of $\mathrm{int}(W_i)$.
Thus, the closed subset
$\Xi_{i}^{-1}\left(\Xi_{i}\left(\mathcal{P}_{k}\right)\right)$
of $\mathcal{F}$ represents
the collection of the closed subsets
$F\cup A$ satisfying $F\cap\mathrm{int}(W_i)=\varnothing$
and $A\subseteq\mathrm{int}(W_i)$ with $|A|\le k$.
Hence, we can identify
the $F_{\sigma\delta}$-set $\mathcal{P}_{\mathrm{lf}}$ of (\ref{p.lf})
with the collection $\mathcal{P}_{\mathrm{lf}}$
of locally finite subsets in Section~\ref{val.intro}.

\begin{remark}\label{separable}
Norberg~\cite{norberg} in his proof of Choquet theorems on domains
assumed the following property for a domain $\mathbb{L}$.
A subset $Q$ of $\mathbb{L}$ is said to be \emph{separating}
if $x \ll y$ implies $x\le q\le y$ for some $q\in Q$.
It is easily observed that if $Q$ is a
separating subset of $\mathbb{L}$ then an open filter $V$ of $\mathbb{L}$ can
be expressed as $V = \bigcup_{q\in Q\cap V}\langle q \rangle^*$.
Therefore,
if a domain $\mathbb{L}$ has a countable separating subset then any open
filter of $\mathbb{L}$ is $\sigma$-compact.
In Example~\ref{class.k}
the domain $\mathcal{K}$ has a countable separating subset
if $R$ is second-countable.
In the setting of Definition~\ref{scott.def}
the lattice $\mathcal{F}$ is second-countable
if $\mathbb{K}$ has a countable separating subset
(Theorem~3.1 of~\cite{norberg}).
\end{remark}

\begin{example}\label{nonsep.ex}
Here we continue Example~\ref{ex.nonsep}.
Let $Q_1$ be the collection of rational numbers on $R_1$.
Then $R$ is not second-countable, having an uncountable open base
consisting of
$G_{x_0,r_1,s_1} = \{(x_0,x_1): x_1\in (r_1-s_1,r_1+s_1)\cap R_1\}$
with $x_0\in R_0$ and $r_1,s_1\in Q_1$.
In order to separate each pair
$\{(x_0,x_1): a_1\le x_1\le b_1\}
\ll \{(x_0,x_1): a_2\le x_1\le b_2\}$
of compact subsets in $R$,
we must have an uncountable separating subset of $\mathcal{K}$.
\end{example}

In the proof of Lemma~\ref{v.approx}
we can observe that
$r(x) = m(\langle x\rangle\setminus\langle B_F^x\rangle)
= \nabla_{B_F^x} \varphi(x)$ for $x \in F$,
and that $\mu_F(\mathcal{P}_k) = \sum r(x)$,
where the summation is over $\{x\in F: 0\le |B_F^x| \le k\}$.
This observation leads to
the following corollary to Lemma~\ref{v.approx}.

\begin{corollary}\label{w.approx}
The discrete measure $\mu_F$ of
Lemma~\ref{v.approx} satisfies
\begin{equation}\label{w.bound}
\mu_F\left(\Xi_i^{-1}(\Xi_i(\mathcal{P}_k))\right)
\ge \sum \nabla_{B_F^x}\varphi(x),
\end{equation}
where the summation is over all $x\in F$
satisfying $0 \le |B_F^x| \le k$.
The equality in (\ref{w.bound}) holds
if $w_i \ll \hat{0}_F$.
\end{corollary}

\begin{definition}\label{lfv.def}
A completely monotone Scott function $\varphi$
is called a \emph{locally finite valuation} if
for any $w\in \mathbb{K}$ and $\delta > 0$
we can find some positive integer $k$ so that
the right-hand side of (\ref{w.bound}) has
the lower bound $(1-\delta)$
whenever a finite sublattice $F$ satisfies $w \le \hat{0}_F$.
\end{definition}

\begin{theorem}\label{lfv.rep}
Let $\varphi$ be a completely monotone Scott function on $\mathbb{K}$,
and let $\mu$ be the representation of $\varphi$ over $\mathcal{F}$.
Then $\varphi$ is a locally finite valuation
if and only if $\mu$ satisfies
$\mu(\mathcal{P}_{\mathrm{lf}}) = 1$.
\end{theorem}

\begin{proof}
Suppose $\mu(\mathcal{P}_{\mathrm{lf}}) = 1$.
Then we have
$\mu\left(
\bigcup_{k=1}^\infty\Xi_{i}^{-1}(\Xi_{i}(\mathcal{P}_{k}))
\right) = 1$ for every $i$.
Let $w\in \mathbb{K}$ and $\delta > 0$ be fixed.
Then we can find $w_i \ll w$ and choose $k$ so that
$\mu\left(\Xi_{i}^{-1}(\Xi_{i}(\mathcal{P}_{k}))\right)
\ge 1 - \delta$.
Let $F$ be a finite sublattice of $\mathbb{K}$ satisfying
$w \le \hat{0}_F$.
If $|B_F^x| \ge k+1$ then
we have
$\mathcal{F}_x^{B_F^x}\cap\Xi_{i}^{-1}(
\Xi_{i}(\mathcal{P}_k))
=\Xi_{i}^{-1}(\Xi_{i}(\mathcal{F}_x^{B_F^x}\cap\mathcal{P}_k))
= \varnothing$.
As we demonstrated in Corollary~\ref{p.k.bound},
we can see that
$\mu\left(\Xi_{i}^{-1}(\Xi_{i}(\mathcal{P}_{k}))\right)$
is the lower bound
for the right-hand side of (\ref{w.bound}),
so $\varphi$ is a locally finite valuation.

Conversely, suppose $\varphi$ is a locally finite valuation.
For each $w_i$ we can find a subnet $\{\mu_{F'}\}$
of $\mu_F$'s constructed in Lemma~\ref{v.approx}
with the index of subnet consisting
of sublattices $F'$ satisfying $w_i \ll\bigwedge F'$.
Let $\mu$ be a limit of the net $\{\mu_{F}\}$,
and let $\mu_i$ be a limit of the subnet $\{\mu_{F'}\}$.
Then we can construct a Radon measure
$\tilde{\mu}_i(\mathcal{U}) = \mu_i(\Xi_{i}^{-1}(\mathcal{U}))$
on Borel-measurable subsets $\mathcal{U}$ of
the compact space $\mathrm{OFilt}(\llangle w_i\rrangle^*)
= \Xi_{i}(\mathcal{F})$,
which represents $\varphi$
restricted to the domain $\llangle w_i\rrangle^*$.
Since $\tilde{\mu}_i$ must be uniquely determined
by Theorem~\ref{m.cm},
we must have
$\mu_i(\Xi_{i}^{-1}(\mathcal{U})) = \mu(\Xi_{i}^{-1}(\mathcal{U}))$.

Let $\delta > 0$ be arbitrary.
By Corollary~\ref{w.approx} and Definition~\ref{lfv.def}
we can find a sufficiently large $k$ so that
$\tilde{\mu}_i\left(\Xi_i(\mathcal{P}_k)\right)
\ge\limsup_{F'}
\mu_{F'}\left(\Xi_{i}^{-1}(\Xi_i(\mathcal{P}_k))\right)
\ge 1-\delta$.
By applying the Lebesgue's convergence theorem
we obtain
\begin{equation*}
  \mu\left(\bigcup_{k=1}^\infty
  \Xi_{i}^{-1}(\Xi_i(\mathcal{P}_{k}))\right)
  = \tilde{\mu}_i\left(
  \bigcup_{k=1}^\infty\Xi_i(\mathcal{P}_{k})
  \right)
  = \lim_{k\to\infty}\tilde{\mu}_i\left(
  \Xi_i(\mathcal{P}_k)
  \right) = 1
\end{equation*}
and
\begin{math}
  \mu(\mathcal{P}_{\mathrm{lf}})
  = \lim_{i\to\infty}\mu\left(\bigcup_{k=1}^\infty
  \Xi_{i}^{-1}(\Xi_i(\mathcal{P}_{k}))
  \right) = 1 .
\end{math}
\end{proof}

To construct a finite sublattice $F$ of $\mathbb{K}$,
we may start with a finite sup-subsemilattice $G$ of $\mathbb{K}$,
and generate the sublattice $F$.
Then the collection $J_F'$ of prime elements of $F$
is contained in $G' = G\setminus\{\bigvee G\}$
(although not necessarily equal to it).
Then the summation of (\ref{w.bound})
can be somewhat simplified as follows.

\begin{proposition}\label{bound.prop}
Let $G$ be a finite sup-subsemilattce of $\mathbb{K}$, and
let $F$ be the sublattice generated by $G$.
Then the right-hand side of (\ref{w.bound}) is equal to
\begin{equation}\label{lfv.bound}
  \sum \nabla_{B}\varphi\left(
  \bigwedge(G\setminus\langle B\rangle)
  \right)
\end{equation}
where the summation is over all the antichains $B$'s
in $G'$ satisfying $0\le |B|\le k$.
\end{proposition}

\begin{proof}
Let $B$ be an antichain in $G'$.
If $B\subseteq J_F$ then we can find $B = B_F^x$
for which we find
$\bigwedge(G\setminus\langle B\rangle) = x$.
Otherwise, some element of $B$
is not $\wedge$-irreducible,
and $F\setminus\langle B\rangle$ must have at least two distinct
minimal elements, say $y_1$ and $y_2$, so that
$y_1\wedge y_2\in\langle B\rangle$.
Since $\bigwedge(G\setminus\langle B\rangle)\le y_1\wedge y_2$,
we must have
$\nabla_{B}\varphi\left(
\bigwedge(G\setminus\langle B\rangle)
\right) = 0$.
\end{proof}

In Proposition~\ref{bound.prop}
we set $o_B = \bigwedge(G\setminus\langle B\rangle)$,
and call it an \emph{opening} if $o_B\not\in\langle B\rangle$.
If $o_B\in\langle B\rangle$ then
$\nabla_{B}\varphi(o_B) = 0$;
thus, the summation of (\ref{lfv.bound}) is over those antichains
$B$'s for which $o_B$ is an opening,
which is equal to
the left-hand side of (\ref{varphi.lfv})
in Section~\ref{val.intro}.
Hence, we have established
Theorem~\ref{val.3} by considering
the nondegenerate locally finite valuation $\varphi(x) = 1 - \Phi(x)$
in Theorem~\ref{lfv.rep}.

\section{L\'{e}vy exponents}
\label{levy.exp}

Throughout this section we assume that
$\mathbb{K}$ has a countable separating subset of Remark~\ref{separable}.
Then we can consider an open filter $\mathbb{V}$ of $\mathbb{K}$,
and set it as the $\sigma$-compact domain.
Let $\Phi$ be a completely alternating conjugate on the domain
$\mathbb{V}$,
and let $\{v_i\}_{i=1}^{\infty}$ be a chain in $\mathbb{V}$
so that $\mathbb{V} = \bigcup_{i=1}^\infty\langle v_i\rangle^*$;
set $v_0 = \hat{1}$ for convenience.
By Theorem~\ref{m.ca}
we can find a unique representation $\lambda$ of $\Phi$
over $\mathcal{V}' = \mathrm{OFilt}(\mathbb{V})\setminus\{\mathbb{V}\}$.
Without loss of generality
we may assume that
$\lambda_i = \Phi(v_{i}) - \Phi(v_{i-1}) > 0$,
and partition $\mathcal{V}'$ into
$\mathcal{V}_{v_{i-1}}^{v_i} = \{W\in\mathcal{V}':
v_{i-1}\in W,\, v_{i}\not\in W\}$
for $i=1,2,\ldots$.
It should be noted that
$\sigma$-compactness is not required when $\Phi$ is bounded,
in which case
the sequence of constant values $\lambda_i$ is replaced by
a single normalizing constant
$\lambda(\mathcal{V}') = \sup_{x\in \mathbb{V}}\Phi(x)$;
see Remark~\ref{m.ca.rem}.
Also, the case of $\lambda(\mathcal{V}') = 0$
[i.e., $\Phi\equiv 0$] is allowed in the following discussion.

Consider
some probability space $(\Omega,\mathcal{B},\mathbb{P})$.
Then we can introduce a Poisson random variable
$N_i$ with parameter $\lambda_i$ for each $i=1,2,\ldots$,
and conditionally given $N_i = n_i$ we can independently sample
random sets $\xi_{i,1},\ldots,\xi_{i,n_i}$ from
the probability  measure $\lambda(\cdot)/\lambda_i$
normalized and restricted on $\mathcal{V}_{v_{i-1}}^{v_i}$.
This collectively forms
a collection of Poisson events $\xi_{i,j}$'s
on the space $\mathcal{V}'$.
We can construct
$\zeta = \bigcap_{i=1}^{\infty}\bigcap_{j=1}^{N_i}\xi_{i,j}$,
where we set $\bigcap_{j=1}^{N_i}\xi_{i,j} = \mathbb{V}$ if $N_i = 0$.
Note that we have $\zeta = \mathbb{V}$ with probability
$e^{-\lambda(\mathcal{V}')}$ if $\Phi$ is bounded.
Suppose $x,y\in\zeta(\omega)$
at $\omega\in\Omega$.
Then we can find some $v_{k}$
satisfying $v_{k} \ll x\wedge y$,
and some $v_{k} \ll z \ll x\wedge y$ satisfying
$\llangle z \rrangle^*\subset
\bigcap_{i=1}^{k}\bigcap_{j=1}^{N_i(\omega)}\xi_{i,j}(\omega)$;
thus $\llangle z \rrangle^*\subset\zeta(\omega)$,
implying that $\zeta(\omega)$ is an open filter.
Thus we obtain a $\mathcal{V}$-valued random variable $\zeta$,
and call $\zeta$ a \emph{$\cap$-compound of Poisson events}
generated by $\Phi$.

\begin{lemma}\label{levy.lem}
Let $\Phi$ be a completely alternating conjugate
on the domain $\mathbb{V}$,
and let $\varphi(x) = \exp[-\Phi(x)]$
if $x\in \mathbb{V}$; otherwise, $\varphi(x) = 0$.
Then the $\cap$-compound $\zeta$ of Poisson events
generated by $\Phi$ represents $\varphi$.
\end{lemma}

\begin{proof}
If $x\not\in \mathbb{V}$ we see trivially
$\mathbb{P}(x\in\zeta) = 0$.
If $x\in \mathbb{V}$ then we obtain
\begin{align}
&
\mathbb{P}(x\in\zeta)
=
\prod_{i=1}^{\infty}
\mathbb{E}\left[\prod_{j=1}^{N_i}\mathbb{P}(x\in\xi_{i,j})\right]
=
\prod_{i=1}^{\infty}
\mathbb{E}\left[
\left(\frac{\lambda\left(\mathcal{V}_{x\wedge v_{i-1}}^{v_i}\right)}
{\lambda_i}\right)^{N_i}\right]
\nonumber
\\ & \hspace{0.2in} =
\prod_{i=1}^{\infty}
\exp[-\lambda(\mathcal{V}_{v_{i-1}}^{v_i}\setminus\mathcal{V}_{x})]
=
\exp[-\lambda(\mathcal{V}^x)]
=
\exp[-\Phi(x)],
\label{levy.rep}
\end{align}
where $\mathbb{E}[X]$ denotes the expectation
of a real-valued random variable $X$
on the probability space $(\Omega,\mathcal{B},\mathbb{P})$.
\end{proof}

We call $\Phi$ the \emph{L\'{e}vy exponent} of $\varphi$
if $\varphi$ of Lemma~\ref{levy.lem}
is obtained from a completely alternating conjugate $\Phi$.
This probabilistic interpretation of $\Phi$ is largely due to
Math\'{e}ron~\cite{matheron} and Norberg~\cite{norberg}
who found the formulation of $\varphi$
in Lemma~\ref{levy.lem} analogous to
L\'{e}vy-Khinchin formula (cf.~\cite{berg,bertoin}).

Consider independent and identically distributed (iid)
$\mathcal{F}$-valued random variables $\xi_1,\ldots,\xi_n$
whose distribution is determined by
$\varphi_{(n)}(x) = \mathbb{P}(x\in\xi_i)$.
Then the $\mathcal{F}$-valued random variable
$\bigcap_{i=1}^n\xi_i$ is distributed as
\begin{math}
\mathbb{P}\left(
x\in\bigcap_{i=1}^n\xi_i\right)
= \prod_{i=1}^n\mathbb{P}(x\in\xi_i)
= \left(\varphi_{(n)}(x)\right)^n.
\end{math}
Conversely, if for any integer $n$
there exists a completely monotone Scott function
$\varphi_{(n)}$ such that $\varphi =
\left(\varphi_{(n)}\right)^n$
then
$\varphi$ is called \emph{infinitely divisible}.

\begin{proposition}\label{inf.div}
A Scott function $\varphi$
is infinitely divisible if and only if
$\varphi$ has a L\'{e}vy exponent $\Phi$.
\end{proposition}

\begin{proof}
Suppose that $\varphi$ is infinitely divisible.
Then we can claim that
the Scott open subset $\mathbb{V} = \{x\in \mathbb{K}: \varphi(x) > 0\}$
is a filter.
If not, we should find a pair $x,y\in \mathbb{V}$ satisfying
$\varphi(x) > 0$, $\varphi(y) > 0$, and $\varphi(x\wedge y) = 0$,
for which the complete monotonicity of $\varphi_{(n)}$
implies
$1 \ge \varphi_{(n)}(x\vee y)
\ge \varphi_{(n)}(x) + \varphi_{(n)}(y) - \varphi_{(n)}(x\wedge y)
= \varphi(x)^{1/n} + \varphi(y)^{1/n}$.
But this cannot hold for sufficiently large $n$, drawing
a contradiction.
Thus, we see that
$$
\Phi(x) = -\ln\varphi(x) =
\lim_{n\to\infty}n[1-\varphi(x)^{1/n}]
= \lim_{n\to\infty}n[1-\varphi_{(n)}(x)]
$$
is a completely alternating conjugate on the domain $\mathbb{V}$.

Conversely suppose that $\Phi$ is the
L\'{e}vy exponent of $\varphi$ on some open filter $\mathbb{V}$.
Then $\Phi/n$ generates
a $\cap$-compound $\zeta_{(n)}$ of Poisson events,
and it is the L\'{e}vy exponent of
$\varphi_{(n)}(x) = \exp[-\Phi(x)/n]$ if $x\in \mathbb{V}$;
otherwise, $\varphi_{(n)}(x) = 0$.
Thus, we find
$\varphi = \left(\varphi_{(n)}\right)^n$.
\end{proof}

In what follows we assume that $\mathbb{K}$ is distributive.
A Scott function $\varphi$ is called an
\emph{exponential valuation} if
$\varphi(x)\varphi(y) = \varphi(x\wedge y)\varphi(x\vee y)$
for every pair $x,y\in \mathbb{K}$.
Then we can immediately observe
the following characterization of exponential valuation.

\begin{proposition}\label{pois.prop}
A Scott function $\varphi$
is a strictly positive exponential valuation
if and only if $\varphi$ has
a L\'{e}vy exponent $\Phi$ which
is a conjugate valuation on $\mathbb{K}$.
\end{proposition}

\begin{proof}
If $\varphi$ is an exponential valuation
then $\Phi(x) = -\ln\varphi(x)$ satisfies
(\ref{module}); thus, it is a conjugate valuation on the domain
$\mathbb{K}$ by Proposition~\ref{valuation}.
The converse is obvious.
\end{proof}

Consider the L\'{e}vy exponent $\Phi$ of
Proposition~\ref{pois.prop},
and choose a sequence $\{v_i\}$ satisfying
$\mathbb{K} = \bigcup_{i=1}^\infty\langle v_i\rangle^*$
and $v_{i+1}\ll v_i$ for $i=1,2,\ldots$.
By Theorem~\ref{k.ca} we can find that
the representation $\lambda$ of $\Phi$ is a Radon measure on
$\mathcal{P}_1'$.
As demonstrated in Lemma~\ref{levy.lem},
we can construct a $\cap$-compound
$\zeta = \bigcap_{i=1}^{\infty}\bigcap_{j=1}^{N_i}\xi_{i,j}$
of Poisson events $\xi_{i,j}$'s
on some probability space $(\Omega,\mathcal{B},\mathbb{P})$.
Since each $\xi_{i,j}$ takes values on
\begin{math}
  \mathcal{P}_{v_{i-1}}^{v_i}
  = \{\mathbb{K}\setminus\langle z\rangle:
  v_i\le z,\,v_{i-1}\not\le z,\, z\in P\},
\end{math}
the $\cap$-compound $\zeta$
takes values on $\mathcal{P}_{\mathrm{lf}}$.
By Theorem~\ref{lfv.rep} we obtain the following corollary to
Proposition~\ref{pois.prop}.

\begin{corollary}\label{pois.cor}
If a Scott function $\varphi$
is a strictly positive exponential valuation
then it is a locally finite valuation.
\end{corollary}

In the context of Theorem~\ref{pois.thm}
we can assume that $R$ is $\sigma$-compact.
Then we find an increasing sequence $W_i$ of compact subsets of
$R$ satisfying $R = \bigcup_{i=1}^\infty W_i$,
and set $W_0 = \varnothing$.
Let $\varphi$ be an exponential valuation
continuous on the right over the class $\mathcal{K}$
of compact subsets of $R$.
By Theorem~\ref{val.2}
(or equivalently by Theorem~\ref{k.ca})
the L\'{e}vy exponent $\Phi$ of
Proposition~\ref{pois.prop} is uniquely represented by
a Radon measure $\lambda$ on $R$.
We can generate Poisson events $\xi_{i,j}$'s of a singleton
from the normalized measure $\lambda$
each restricted on $W_i\setminus W_{i-1}$,
and define the number $N(Q)$ of
Poisson events $\xi_{i,j}$'s satisfying $\xi_{i,j}\in Q$.
Then it becomes a routine argument to show that
$N(Q)$ has a Poisson distribution with parameter $\lambda(Q)$,
and that $N(Q_1),\ldots,N(Q_k)$ are independent
for any disjoint sequence of $Q_1,\ldots,Q_k$.
Therefore, we have constructed a general Poisson process.
Conversely, the zero probability function
$\varphi(Q) = \mathbb{P}(N(Q) = 0) = e^{-\lambda(Q)}$
is a strictly positive exponential valuation,
which completes the proof of Theorem~\ref{pois.thm}.

\end{document}